\documentclass{amsart}
\usepackage{graphicx}

\usepackage{amsfonts}
\usepackage{amssymb}
\usepackage{amsmath}
\usepackage{amsthm}
\usepackage{amscd}

\def\vbar{\mathchoice{\vrule height6.3ptdepth-.5ptwidth.8pt\kern- .8pt}
{\vrule height6.3ptdepth-.5ptwidth.8pt\kern-.8pt} {\vrule
height4.1ptdepth-.35ptwidth.6pt\kern-.6pt} {\vrule
height3.1ptdepth-.25ptwidth.5pt\kern-.5pt}}
\def\fudge{\mathchoice{}{}{\mkern.5mu}{\mkern.8mu}}
\def\bbc#1#2{{\rm \mkern#2mu\vbar\mkern-#2mu#1}}
\def\bbb#1{{\rm I\mkern-3.5mu #1}}
\def\bba#1#2{{\rm #1\mkern-#2mu\fudge #1}}
\def\bb#1{{\count4=`#1 \advance\count4by-64 \ifcase\count4\or\bba
A{11.5}\or \bbb B\or\bbc C{5}\or\bbb D\or\bbb E\or\bbb F \or\bbc
G{5}\or\bbb H\or \bbb I\or\bbc J{3}\or\bbb K\or\bbb L \or\bbb
M\or\bbb N\or\bbc O{5} \or \bbb P\or\bbc C{5}\or\bbb B\or\bbc
S{4.2}\or\bba T{10.5}\or\bbc U{5}\or \bba V{12}\or\bba
W{16.5}\or\bba X{11}\or\bba Y{11.7}\or\bba Z{7.5}\fi}}

\newcommand{\A}{{\mathcal{A}}}

\newtheorem{df}{Definition}[section]
\newtheorem{thm}{Theorem}[section]
\newtheorem{cor}{Corollary}[section]
\newtheorem{rem}{Remark}[section]

\newtheorem{prop}{Proposition}[section]
\newtheorem{exa}{Example}[section]
\newtheorem{lem}{Lemma}[section]

\begin{document}
\date{}
\title{Constructions and Cohomology of color Hom-Lie algebras}
\author{K. Abdaoui, F. Ammar and A. Makhlouf }
\address{K. Abdaoui and  F. Ammar, University of Sfax, Faculty of Sciences Sfax,  BP
1171, 3038 Sfax, Tunisia. }
\email{Abdaouielkadri@hotmail.com}
\email{Faouzi.Ammar@fss.rnu.tn}
\address{A. Makhlouf, University of Haute Alsace, 4 rue des fr\`eres Lumi\`ere, 68093 Mulhouse France. }\email{Abdenacer.Makhlouf@uha.fr}
 \maketitle{}
 \begin{abstract}The main purpose of this paper is to define   representations and a cohomology of color Hom-Lie algebras and to study some key constructions and properties. We describe   Hartwig-Larsson-Silvestrov Theorem in the case of $\Gamma$-graded algebras, study  one-parameter formal deformations, discuss  $\alpha^{k}$-generalized derivation and provide examples.
 \end{abstract}
 \begin{center}{\textbf{Introduction}}\end{center} Color Hom-Lie algebras are the natural generalizations of Hom-Lie algebras and Hom-Lie superalgebras. In recent years, they have become an interesting subject of mathematics and physics. A cohomology  of color Lie algebras were introduced and investigated, see \cite{scheunert1979theory,scheunert1979generalized},  and the representations of color Lie algebras were explicitly described in \cite{feldvoss2001represcolor}. As is  well known, derivations and extensions of Hom-Lie algebras, Hom-Lie superalgebras and color Hom-Lie algebras are very important subjects. Color Hom-Lie algebras were studied in \cite{yuan2010hom}. In the particular case of Hom-Lie superalgebras, cohomology theory of  was provided in \cite{sadaoui2011cohomology}.\\
This paper focusses on $\Gamma$-graded Hom-algebras with $\Gamma$ is an abelian group. Mainly, we prove a $\Gamma$-graded version of a Hartwig-Larsson-Silvestrov Theorem and we study representations and cohomology of color Hom-Lie algebras.\\
The paper is organized as follows. In section 1, we recall  definitions and some key constructions of color Hom-Lie algebras and we provide a list of twists of color Hom-Lie algebras. Section $2$ is dedicated to describe and prove the $\Gamma$-graded version Hartwig-Larsson-Silvestrov Theorem, which was proved for Hom-Lie algebras in  \cite[Theorem $5$]{hartwig2006deformations} and  for Hom-Lie superalgebras in \cite[Theorem $4.2$]{ammar2010hom}$)$. In section $3$, we construct a family of cohomologies  of color Hom-Lie algebras, discuss representation theory in connection with cohomology  and compute the second cohomology group of $(sl_{2}^{c},[.,.]_{\alpha},\varepsilon,\alpha)$. In section $4$, we study  formal deformations of color Hom-Lie algebras. In last section, we study the homogeneous $\alpha^{k}$-generalized derivations and the $\alpha^{k}$-centroid of color Hom-Lie algebras and we give some properties generalizing the homogeneous generalized derivations discussed in \cite{LIn ni}. Moreover in Proposition $\ref{DERhomJORDAN}$ we prove that the $\alpha$-derivation of color Hom-Lie algebras gives rise to a color  Hom-Jordan algebras.

\begin{center}\section{Definitions, Constructions and Examples}\end{center}
In the following we summarize definitions of color Hom-Lie and color Hom-associative algebraic
structures (see \cite{yuan2010hom}) generalizing the well known color Lie and color associative algebras.\\
Throughout the article we let $\mathbb{K}$ be an algebraically closed field of characteristic $0$ and
$\mathbb{K^{\ast}}=\mathbb{K}\backslash \{0\}$ be the group of units of $\mathbb{K}$.\\
$\bullet$~~Let $\Gamma$ be an abelian group. A vector space $V$ is said to be $\Gamma$-graded, if there is a family
$(V_{\gamma})_{\gamma\in \Gamma}$ of vector subspace of $V$ such that $$V=\bigoplus_{\gamma\in \Gamma}V_{\gamma}.$$
An element $x \in V$ is said to be homogeneous of the degree $\gamma \in \Gamma$ if $x \in V_{\gamma}, \gamma\in \Gamma$, and in this case, $\gamma$ is called the color of $x$. As usual, denote by $\overline{x}$ the color of an element $x \in V$. Thus each homogeneous
element $x \in V$ determines a unique group of element  $\overline{x} \in \Gamma$ by $x \in V_{\overline{x}}$. Fortunately, We can almost drop
the symbol $"-"$, since confusion rarely occurs.  In the sequel, we will denote by $\mathcal{H}(V)$ the set of all the homogeneous elements of $V$.\\
Let $V=\bigoplus_{\gamma\in \Gamma}V_{\gamma}$ and $V^{'}=\bigoplus_{\gamma\in \Gamma}V^{'}_{\gamma}$ be two $\Gamma$-graded vector spaces. A linear mapping $f: V \longrightarrow V^{'}$  is said to be homogeneous of the degree $\upsilon \in \Gamma$ if\\
  $$f(V_{\gamma})\subseteq V^{'}_{\gamma+\upsilon}, \forall~~ \gamma \in \Gamma.$$
if in addition, $f$ is  homogeneous of degree zero, i.e. $f(V_{\gamma})\subseteq V^{'}_{\gamma}$ holds for any $\gamma \in \Gamma$, then $f$ is said to be even.\\
$\bullet$~~An algebra $\mathcal{A}$ is said to be $\Gamma$-graded if its underlying vector space is $\Gamma$-graded, i.e. $\mathcal{A}=\bigoplus_{\gamma\in \Gamma}\mathcal{A}_{\gamma}$, and if, furthermore $\mathcal{A}_{\gamma}\mathcal{A}_{\gamma'}\subseteq \mathcal{A}_{\gamma+\gamma'}$, for all $\gamma, \gamma'\in \Gamma$. It is easy to see
that if $\mathcal{A}$ has a unit element $e$, it follows $e \in \mathcal{A}_{0}$. A subalgebra of $\mathcal{A}$ is said to be graded if its graded as a subspace of $\mathcal{A}$.\\
Let $\mathcal{A}^{'}$ be another $\Gamma$-graded algebra. A homomorphism $f:\mathcal{A} \longrightarrow \mathcal{A}^{'}$ of $\Gamma$-graded algebras is by definition a homomorphism  of the algebra $\mathcal{A}$ into the algebra  $\mathcal{A}^{'}$, which is, in addition an even mapping.\\
\begin{df} Let $\mathbb{K}$ be a field and  $\Gamma$ be an abelian group. A map $\varepsilon:\Gamma\times\Gamma\rightarrow \mathbb{K^{\ast}}$ is called a skew-symmetric \textit{bi-character} on ${\Gamma}$ if the following identities hold , for all $a,b,c$ in $\Gamma$
\begin{enumerate}
\item $\varepsilon(a,b)~\varepsilon(b,a)=1,$
\item $\varepsilon(a,b+c)=\varepsilon(a,b)~\varepsilon(a,c),$
\item $\varepsilon(a+b,c)=\varepsilon(a,c)~\varepsilon(b,c).$
\end{enumerate}
\end{df}
 The definition above implies, in particular, the following relations
$$\varepsilon(a,0)=\varepsilon(0,a)=1,\ \varepsilon(a,a)=\pm1, \  \textrm{for\ all}\  a \in \Gamma.$$
If $x$ and $x'$ are two homogeneous elements of degree $\gamma$ and $\gamma'$ respectively and $\varepsilon$ is a skew-symmetric bi-character, then we shorten the notation by writing $\varepsilon(x,x')$ instead of $\varepsilon(\gamma,\gamma')$.
\begin{df}\cite{yuan2010hom} A \textit{color Hom-Lie} algebra is a quadruple $(\mathcal{A},[.,.],\varepsilon,\alpha)$ consisting of a $\Gamma$-graded vector space $\mathcal{A}$, a bi-character $\varepsilon$, an even bilinear mapping
$[.,.]:\mathcal{A}\times\mathcal{A}\rightarrow\mathcal{A}$ $($i.e. $[\mathcal{A}_{a},\mathcal{A}_{b}]\subseteq \mathcal{A}_{a+b}$ for all $a,b \in \Gamma)$ and an even homomorphism $\alpha:\mathcal{A}\rightarrow\mathcal{A}$ such that for homogeneous elements we have
$$\label{sksymmetric}[x,y]=-\varepsilon(x,y)[y,x] ~~(\varepsilon\textrm{-skew\ symmetric}).$$
$$\label{HJCI}\circlearrowleft_{x,y,z}\varepsilon(z,x)[\alpha(x),[y,z]]=0~~(\varepsilon\textrm{-Hom-Jacobi\ condition}).$$
where $\circlearrowleft_{x,y,z}$ denotes summation over the cyclic permutation on $x,y,z$.
\end{df}
In particular, if $\alpha$ is a morphism of color Lie algebras $($i.e. $\alpha\circ[.,.]=[.,.]\circ\alpha^{\otimes2})$, then we call $(\mathcal{A},[.,.],\varepsilon,\alpha)$ a multiplicative color Hom-Lie algebra.\\
Observe that when $\alpha=Id$, the $\varepsilon$-Hom-Jacobi condition $(\ref{HJCI})$ reduces to the usual $\varepsilon$-Jacobi condition\\
$$\circlearrowleft_{x,y,z}\varepsilon(z,x)[x,[y,z]]=0$$
for all $x,y,z \in \mathcal{H}(\mathcal{A}).$
\begin{exa} A color Lie algebra $(\mathcal{A},[.,.],\varepsilon)$ is a color Hom-Lie algebra with $\alpha=Id$, since the $\varepsilon$-Hom-Jacobi condition
reduces to the $\varepsilon$-Jacobi condition when $\alpha=Id$.
\end{exa}
\begin{df} Let $(\mathcal{A},[.,.],\varepsilon,\alpha)$ be a color Hom-Lie algebra. It is called
\begin{enumerate}
\item \textsf{multiplicative} color Hom-Lie algebra if $\alpha[x,y]=[\alpha(x),\alpha(y)]$.
\item \textsf{regular} color Hom-Lie algebra if $\alpha$ is an automorphism.
\item \textsf{involutive} color Hom-Lie algebra if $\alpha$ is an involution, that is $\alpha^{2}=Id$.
\end{enumerate}
\end{df}
Let $(\mathcal{A},[.,.],\varepsilon,\alpha)$ be a regular color Hom-Lie algebra. It was observed in \cite{Gohr2009hom} that the composition method using $\alpha^{-1}$ leads to a color Lie algebra.
\begin{prop} Let $(\mathcal{A},[.,.],\varepsilon,\alpha)$ be a regular color Hom-Lie algebra. Then $(\mathcal{A},[.,.]_{\alpha^{-1}}=\alpha^{-1}\circ [.,.],\varepsilon)$ is a color Lie algebra.
\end{prop}

\begin{proof} Note that $[.,.]_{\alpha^{-1}}$ is $\varepsilon$-skew-symmetric because $[.,.]$  is $\varepsilon$-skew-symmetric and $\alpha^{-1}$ is linear.\\
For  $x,y,z \in \mathcal{H}(\mathcal{A})$, we have$:$
\begin{eqnarray*}
  \circlearrowleft_{x,y,z}\varepsilon(z,x)[x,[y,z]_{\alpha^{-1}}]_{\alpha^{-1}}
  &=&\circlearrowleft_{x,y,z}\varepsilon(z,x)\alpha^{-1}[x,[y,z]_{\alpha^{-1}}]  \\
   &=& \alpha^{-2}(\circlearrowleft_{x,y,z}\varepsilon(z,x)[\alpha(x),[y,z]]) \\
   &=&0.
\end{eqnarray*}
\end{proof}
\begin{rem} In particular the proposition is valid when $\alpha$ is an involution.
The following theorem generalize the result of \cite{yuan2010hom}. In the following  starting from a  color Hom-Lie algebra and an even color Lie algebra endomorphism, we construct a new color Hom-Lie algebra.
\end{rem}
We recall in the following the definition of Hom-associative color algebra which provide a different way for constructing color Hom-Lie algebra by extending the fundamental construction of color Lie algebras from associative color algebra via commutator bracket multiplication. This structure was introduced by \cite{yuan2010hom} .
\begin{df}\cite{yuan2010hom} A \textit{color Hom-associative} algebra is a triple $(\mathcal{A},\mu,\alpha)$ consisting of a $\Gamma$-graded linear space $\mathcal{A}$, an even bilinear map $\mu:\mathcal{A}\times \mathcal{A}\rightarrow \mathcal{A}$ $($i.e $\mu(\mathcal{A}_a,\mathcal{A}_b)\subset \mathcal{A}_{a+b})$ and an even homomorphism $\alpha:\mathcal{A}\rightarrow \mathcal{A}$ such that
\begin{equation}\label{HACA}
    \mu(\alpha(x),\mu(y,z))=\mu(\mu(x,y),\alpha(z)).
\end{equation}
In the case where $\mu(x,y)=\varepsilon(x,y)\mu(y,x)$, we call the Hom-associative color algebra $(\mathcal{A},\mu,\alpha)$ a commutative Hom-associative color algebra.
\end{df}
\begin{rem} We recover classical associative color algebra when $\alpha=Id_{\mathcal{A}}$ and the condition $(\ref{HACA})$ is the associative condition in this case.
\end{rem}
\begin{prop}\cite{yuan2010hom} Let $(\mathcal{A},\mu,\alpha)$ be a Hom-associative color algebra defined on the linear space $\mathcal{A}$ by the multiplication $\mu$ and an even homomorphism $\alpha$. Then the quadruple $(\mathcal{A},[.,.],\varepsilon,\alpha)$, where the bracket is defined for $x, y \in \mathcal{H}(\mathcal{A})$ by
$$[x,y]=\mu(x,y)-\varepsilon(x,y)\mu(y,x)$$
is a color Hom-Lie algebra.
\end{prop}
\begin{thm}\label{induced-color} Let $(\mathcal{A},[.,.],\varepsilon,\alpha)$ be a color Hom-Lie algebra and $\beta:\mathcal{A}\longrightarrow\mathcal{A}$ be an even color Lie algebra endomorphism. Then $(\mathcal{A},[.,.]_{\beta},\varepsilon,\beta\circ\alpha)$, where $[x,y]_{\beta}=\beta\circ [x,y]$, is a color Hom-Lie algebra.\\
Moreover, suppose that $(\mathcal{A}^{'},[.,.]^{'},\varepsilon)$ is a color Lie algebra and $\alpha^{'}:\mathcal{A}^{'}\longrightarrow\mathcal{A}^{'}$ is color Lie algebra endomorphism. If $f:\mathcal{A}\longrightarrow\mathcal{A}^{'} $ is a Li color algebra morphism that satisfies $f \circ \beta= \alpha^{'}\circ f$ then
$$f:(\mathcal{A},[.,.]_{\beta},\varepsilon,\beta\circ\alpha)\longrightarrow (\mathcal{A}^{'},[.,.]^{'},\varepsilon,\alpha^{'}) $$
is a morphism of color Hom-Lie algebras.
\end{thm}
\begin{proof} Obviously $[.,.]_{\beta}$ is $\varepsilon$-skew-symmetric and we show that $(\mathcal{A},[.,.]_{\beta},\varepsilon,\beta\circ\alpha)$ satisfies the $\varepsilon$-Hom-Jacobi condition $\ref{HJCI}$. Indeed
\begin{eqnarray*}
   \circlearrowleft_{x,y,z}\varepsilon(z,x)[\beta\circ\alpha(x),[y,z]_{\beta}]_{\beta} &=& \circlearrowleft_{x,y,z}\varepsilon(z,x)[\beta\circ\alpha(x),\beta([y,z])]_{\beta} \\
   &=& \beta^{2}(\circlearrowleft_{x,y,z}\varepsilon(z,x)[\alpha(x),[y,z]]) \\
   &=&  0.
\end{eqnarray*}
The second assertion follows from
$$f([x,y]_{\beta})=f([\beta(x),\beta(y)])=[f\circ \beta(x),f\circ \beta(y)]^{'}=[\alpha^{'}\circ f(x),\alpha^{'}\circ f(y)]^{'}=[f(x),f(y)]^{'}_{\alpha^{'}}.$$
\end{proof}
\begin{exa}Let $(\mathcal{A},[.,.],\varepsilon)$ be a color Lie algebra and $\alpha$ be a color Lie algebra morphism, then $(\mathcal{A},[.,.]_{\alpha}=\alpha \circ [.,.],\varepsilon,\alpha)$ is a multiplicative color Hom-Lie algebra.
\end{exa}
\begin{df} Let $(\mathcal{A},[.,.],\varepsilon,\alpha)$ be a multiplicative color Hom-Lie algebra and $n\geq 0$. Define the nth derived Hom-algebra of $\mathcal{A}$ by
$$\mathcal{A}^{(n)}=(\mathcal{A},[.,.]^{(n)}=\alpha^{2^{n}-1}\circ[.,.],\varepsilon,\alpha^{2^{n}}).$$
Note that $\mathcal{A}^{(0)}=\mathcal{A},~~\mathcal{A}^{(1)}=(\mathcal{A},[.,.]^{(1)}=\alpha\circ[.,.],\varepsilon,\alpha^{2})$, and $\mathcal{A}^{(n+1)}=(\mathcal{A}^{n})^{1}$.
\end{df}
\begin{cor}\label{abdaw} Let $(\mathcal{A},[.,.],\varepsilon,\alpha)$ be a color Hom-Lie algebra. Then the nth derived Hom-algebra of $\mathcal{A}$
$$\mathcal{A}^{(n)}=(\mathcal{A},[.,.]^{(n)}=\alpha^{2^{n}-1}\circ[.,.],\varepsilon,\alpha^{2^{n}})$$
is also a color Hom-Lie algebra for each $n\geq 0$.
\end{cor}
\subsection{Examples of twists of color Hom-Lie algebras}
In this section we provide examples of color Hom-Lie algebras. We use  the classification of color Lie algebra provided  in \cite{silvestrov} and the twisting principle.\\
In Examples $\ref{class1}$ and $\ref{class2}$, $\Gamma$ is the group $\mathbb{Z}_{2}^{3}$, and $\varepsilon: \Gamma \times \Gamma \longrightarrow \mathbb{C}$ is defined by the matrix
$$[\varepsilon(i,j)]=\left(
    \begin{array}{ccc}
       1 & -1 & -1 \\
      -1 &  1 & -1 \\
      -1 & -1 &  1 \\
    \end{array}
  \right).$$
Elements of this matrix specify in the natural way values of $\varepsilon$ on the set $\{(1,1,0),(1,0,1),(0,1,1)\}\times \{(1,1,0),(1,0,1),(0,1,1)\}\subset \mathbb{Z}_{2}^{3}\times \mathbb{Z}_{2}^{3}$ $($The elements $(1,1,0),(1,0,1),(0,1,1)$ are ordered and numbered by the numbers $1,2,3$ respectively$)$. The values of $\varepsilon$ on other elements from $\mathbb{Z}_{2}^{3}\times \mathbb{Z}_{2}^{3}$ do not affect the multiplication $[.,.]$.
\begin{exa}\label{class1} The graded analogue of $sl(2,\mathbb{C})$ is defined as complex algebra with three generators $e_{1},e_{2}$ and $e_{3}$ satisfying the commutation relations $e_{1}e_{2}+e_{2}e_{1}=e_{3},~ e_{1}e_{3}+e_{3}e_{1}=e_{2},~ e_{2}e_{3}+e_{3}e_{2}=e_{1}.$
Let $sl(2,\mathbb{C})$ be $\mathbb{Z}_{2}^{3}$-graded linear space $sl(2,\mathbb{C})=sl(2,\mathbb{C})_{(1,1,0)}~\oplus sl(2,\mathbb{C})_{(1,0,1)}~~\oplus sl(2,\mathbb{C})_{(0,1,1)}$ with basis $e_{1} \in sl(2,\mathbb{C})_{(1,1,0)},~e_{2} \ sl(2,\mathbb{C})_{(1,0,1)},~~e_{3} \in sl(2,\mathbb{C})_{(0,1,1)}$. The homogeneous subspaces of $sl(2,\mathbb{C})$ graded by the elements of $\mathbb{Z}_{2}^{3}$ different from $(1,1,0),(1,0,1)$ and $(1,0,1)$ are zero and so are omitted. The bilinear multiplication $[.,.]:sl(2,\mathbb{C})\times sl(2,\mathbb{C}) \longrightarrow sl(2,\mathbb{C})$ defined, with respect to  the basis  $\{e_{1},e_{2},e_{3}\}$, by the formulas
 \begin{eqnarray*}
&&\ \ [e_{1},e_{1}]=e_{1}e_{1}-e_{1}e_{1}=0,~~ [e_{1},e_{2}]=e_{1}e_{2}+e_{2}e_{1}=e_{3},\\
&&\ \ [e_{2},e_{2}]=e_{2}e_{2}-e_{2}e_{2}=0,~~ [e_{1},e_{3}]=e_{1}e_{3}+e_{3}e_{1}=e_{2},\\
&&\ \ [e_{3},e_{3}]=e_{3}e_{3}-e_{3}e_{3}=0,~~ [e_{2},e_{3}]=e_{2}e_{3}+e_{3}e_{2}=e_{1},
\end{eqnarray*}
makes $sl(2,\mathbb{C})$ into a three-dimensional color-Lie algebra.\\
By using \cite[Theorem 3.14]{yuan2010hom}, we provide this Lie color algebra with a Hom structure, for that we consider a linear even map $\alpha:sl(2,\mathbb{C})\longrightarrow sl(2,\mathbb{C})$ checking $\alpha[x,y]=[\alpha(x),\alpha(y)],~~\forall~~x,y \in H(sl(2,\mathbb{C}))$ in such way $(sl(2,\mathbb{C}),[.,.]_{\alpha}=\alpha\circ[.,.],\alpha)$ is a color Hom-Lie algebra.
The morphism of $sl^{c}_{2}$ are given with respect to the basics $\{e_{1},e_{2},e_{2}\}$ by
\begin{eqnarray*}
\alpha(e_{1})&=&a_{1}e_{1}+a_{2}e_{2}+a_{3}e_{3},\\
\alpha(e_{2})&=&b_{1}e_{1}+b_{2}e_{2}+b_{3}e_{3},\\
\alpha(e_{3})&=&c_{1}e_{1}+c_{2}e_{2}+c_{3}e_{3},
\end{eqnarray*}
where $a_{i}, b_{i},c_{i} \in \mathbb{C}$.\\
Then, we obtain the following lists  of twisted color Hom-Lie algebra $(sl^{c}_{2},[.,.]_{\alpha}=\alpha \circ [.,.],\varepsilon,\alpha)$ in Table $1$:
\begin{center}{Table 1}\end{center}
$$\begin{tabular}{|c|c|c|c|}
  \hline
   $[e_{1},e_{2}]_{\alpha_{1}}=-e_{3}$ && $[e_{1},e_{2}]_{\alpha_{2}}=-e_{3}$&\\
   $[e_{1},e_{3}]_{\alpha_{1}}=-e_{2}$   &$\alpha_{1}=\left(
    \begin{array}{ccc}
       -1 & 0  & 0 \\
        0 & -1 & 0 \\
        0 & 0  & 1 \\
    \end{array}
  \right)$  &$ \hskip-0.3cm[e_{1},e_{3}]_{\alpha_{2}}=e_{2}$ &$\alpha_{2}=\left(
    \begin{array}{ccc}
      -1 & 0 & 0 \\
       0 & 1 & 0 \\
       0 & 0 & -1 \\
    \end{array}
  \right) $\\
   $\hskip-0.3cm[e_{2},e_{3}]_{\alpha_{1}}=e_{1}$ & &  $[e_{2},e_{3}]_{\alpha_{2}}=-e_{3}$ & \\
  \hline
  $[e_{1},e_{2}]_{\alpha_{3}}=-e_{3}$ & & $\hskip-0.3cm[e_{1},e_{2}]_{\alpha_{4}}=e_{3}$&\\
  $\hskip-0.3cm[e_{1},e_{3}]_{\alpha_{3}}=e_{2}$  &$\alpha_{3}=\left(
    \begin{array}{ccc}
        1 & 0  & 0 \\
        0 & -1 & 0 \\
        0 & 0 & -1 \\
    \end{array}
  \right)$ & $[e_{1},e_{3}]_{\alpha_{4}}=-e_{2}$ & $\alpha_{4}=\left(
    \begin{array}{ccc}
       1  & 0  & 0 \\
       0  & 1  & 0 \\
       0  & 0  & 1 \\
    \end{array}
  \right)$ \\
  $[e_{2},e_{3}]_{\alpha_{3}}=-e_{1}$& &$[e_{2},e_{3}]_{\alpha_{4}}=-e_{1}$&  \\ \hline
  $\hskip-0.3cm[e_{1},e_{2}]_{\alpha_{5}}=e_{2}$ & & $[e_{1},e_{2}]_{\alpha_{6}}=-e_{2}$&\\
  $\hskip-0.3cm[e_{1},e_{3}]_{\alpha_{5}}=e_{3}$ &$\alpha_{5}=\left(
    \begin{array}{ccc}
       -1 & 0 & 0 \\
       0 & 0 & 1 \\
       0 & -1 & 0 \\
    \end{array}
  \right)$  & $[e_{1},e_{3}]_{\alpha_{6}}=-e_{3}$ &$\alpha_{6}=\left(
    \begin{array}{ccc}
       -1 & 0  & 0 \\
       0 &  0  & -1 \\
       0 &  1  & 0 \\
    \end{array}
  \right)$  \\
  $\hskip-0.3cm[e_{2},e_{3}]_{\alpha_{5}}=e_{1}$ && $\hskip-0.3cm[e_{2},e_{3}]_{\alpha_{6}}=e_{1}$&\\ \hline
  $[e_{1},e_{2}]_{\alpha_{7}}=-e_{2}$&& $\hskip-0.3cm[e_{1},e_{2}]_{\alpha_{8}}=e_{2}$&\\
  $\hskip-0.3cm[e_{1},e_{3}]_{\alpha_{7}}=e_{3}$& $\alpha_{7}=\left(
    \begin{array}{ccc}
       1 & 0  & 0 \\
       0 & 0  & -1 \\
       0 & -1 & 0 \\
    \end{array}
  \right)$  &$[e_{1},e_{3}]_{\alpha_{8}}=-e_{3}$  & $\alpha_{8}=\left(
    \begin{array}{ccc}
       1 & 0 & 0 \\
       0 & 0 & 1 \\
       0 & 1 & 0 \\
    \end{array}
  \right)$ \\
  $[e_{2},e_{3}]_{\alpha_{7}}=-e_{1}$& &$[e_{2},e_{3}]_{\alpha_{8}}=-e_{1}$&\\\hline
  $\hskip-0.3cm[e_{1},e_{2}]_{\alpha_{9}}=e_{3}$& & $[e_{1},e_{2}]_{\alpha_{10}}=-e_{3}$&\\
 $\hskip-0.3cm[e_{1},e_{3}]_{\alpha_{9}}=e_{1}$ &$\alpha_{9}=\left(
    \begin{array}{ccc}
        0 & -1 & 0 \\
       -1 & 0  & 0 \\
        0 & 0  & 1 \\
    \end{array}
  \right)$ & $[e_{1},e_{3}]_{\alpha_{10}}=-e_{1}$ & $\alpha_{10}=\left(
    \begin{array}{ccc}
        0 & 1 & 0 \\
       -1 & 0 & 0 \\
        0 & 0 & -1 \\
    \end{array}
  \right)$ \\
  $\hskip-0.3cm[e_{2},e_{3}]_{\alpha_{9}}=e_{2}$&&$\hskip-0.3cm[e_{2},e_{3}]_{\alpha_{10}}=e_{2}$&\\\hline
  $\hskip-0.3cm[e_{1},e_{2}]_{\alpha_{11}}=e_{1}$& & $[e_{1},e_{2}]_{\alpha_{12}}=-e_{3}$&\\
  $[e_{1},e_{3}]_{\alpha_{11}}=-e_{2}$&$\alpha_{11}=\left(
    \begin{array}{ccc}
        0 & 0 &1\\
       0 & 1 & 0\\
        1 & 0& 0\\
    \end{array}
  \right)$ &$\hskip-0.3cm[e_{1},e_{3}]_{\alpha_{12}}=e_{1}$  & $\alpha_{12}=\left(
    \begin{array}{ccc}
        0 & -1 & 0 \\
        1 & 0 & 0 \\
        0 & 0 & -1 \\
    \end{array}
  \right)$ \\
  $[e_{2},e_{3}]_{\alpha_{11}}=-e_{2}$&&$[e_{2},e_{3}]_{\alpha_{12}}=-e_{2}$&\\\hline
\end{tabular}$$

$$\begin{tabular}{|c|c|c|c|}
  \hline
  $\hskip-0.3cm[e_{1},e_{2}]_{\alpha_{13}}=e_{3}$& &$\hskip-0.3cm[e_{1},e_{2}]_{\alpha_{14}}=e_{1}$&\\
  $[e_{1},e_{3}]_{\alpha_{13}}=-e_{1}$&$\alpha_{13}=\left(
    \begin{array}{ccc}
       0 & 1 & 0 \\
       1 & 0  & 0 \\
       0 & 0  & 1 \\
    \end{array}
  \right)$ & $\hskip-0.3cm[e_{1},e_{3}]_{\alpha_{14}}=e_{3}$ & $\alpha_{14}=\left(
    \begin{array}{ccc}
       0 & 0 & 1 \\
       -1 & 0 & 0 \\
       0 & -1 & 0 \\
    \end{array}
  \right)$ \\
  $[e_{2},e_{3}]_{\alpha_{13}}=-e_{2}$&&$\hskip-0.3cm[e_{2},e_{3}]_{\alpha_{14}}=e_{2}$&\\\hline
  $[e_{1},e_{2}]_{\alpha_{15}}=-e_{1}$& &$[e_{1},e_{2}]_{\alpha_{16}}=-e_{1}$&\\
   $[e_{1},e_{3}]_{\alpha_{15}}=-e_{3}$&$\alpha_{15}=\left(
    \begin{array}{ccc}
       0 & 0  & -1 \\
       -1 & 0  & 0 \\
       0 & 1 & 0 \\
    \end{array}
  \right)$ & $\hskip-0.3cm[e_{1},e_{3}]_{\alpha_{16}}=e_{3}$ & $\alpha_{16}=\left(
    \begin{array}{ccc}
       0 & 0 & -1 \\
       1 & 0 & 0 \\
       0 & -1 & 0 \\
    \end{array}
  \right)$ \\
  $\hskip-0.3cm[e_{2},e_{3}]_{\alpha_{15}}=e_{2}$&&$[e_{2},e_{3}]_{\alpha_{16}}=-e_{2}$&\\\hline
  $\hskip-0.3cm[e_{1},e_{2}]_{\alpha_{17}}=e_{1}$& &$\hskip-0.3cm[e_{1},e_{2}]_{\alpha_{18}}=e_{2}$&\\
  $[e_{1},e_{3}]_{\alpha_{17}}=-e_{3}$&$\alpha_{17}=\left(
    \begin{array}{ccc}
        0 & 0 & 1 \\
        1 & 0  & 0 \\
       0 & 1  & 0 \\
    \end{array}
  \right)$  & $\hskip-0.3cm[e_{1},e_{3}]_{\alpha_{18}}=e_{1}$ & $\alpha_{18}=\left(
    \begin{array}{ccc}
        0 & -1 & 0 \\
        0 & 0 & 1 \\
       -1 & 0 & 0 \\
    \end{array}
  \right)$ \\
  $[e_{2},e_{3}]_{\alpha_{17}}=-e_{2}$&&$\hskip-0.3cm[e_{2},e_{3}]_{\alpha_{18}}=e_{3}$&\\\hline
   $[e_{1},e_{2}]_{\alpha_{19}}=-e_{2}$& &$[e_{1},e_{2}]_{\alpha_{20}}=-e_{2}$&\\
   $[e_{1},e_{3}]_{\alpha_{19}}=-e_{1}$&$\alpha_{19}=\left(
    \begin{array}{ccc}
        0 &  1 & 0 \\
        0 & 0 & -1 \\
       -1 &  0 & 0 \\
    \end{array}
  \right)$ &$\hskip-0.3cm[e_{1},e_{3}]_{\alpha_{20}}=e_{1}$  & $\alpha_{20}=\left(
    \begin{array}{ccc}
       0 & -1 & 0 \\
       0 & 0 &  -1 \\
       1 & 0 &  0 \\
    \end{array}
  \right)$ \\
  $\hskip-0.3cm[e_{2},e_{3}]_{\alpha_{19}}=e_{3}$&&$[e_{2},e_{3}]_{\alpha_{20}}=-e_{3}$&\\\hline
  $\hskip-0.3cm[e_{1},e_{2}]_{\alpha_{21}}=e_{2}$& & $\hskip-0.3cm[e_{1},e_{2}]_{\alpha_{22}}=e_{1}$&\\
  $[e_{1},e_{3}]_{\alpha_{21}}=-e_{1}$&$\alpha_{21}=\left(
    \begin{array}{ccc}
       0 & 1 & 0 \\
       0 & 0  & 1 \\
       1 & 0  & 0 \\
    \end{array}
  \right)$ & $\hskip-0.3cm[e_{1},e_{3}]_{\alpha_{22}}=e_{2}$ &$\alpha_{22}=\left(
    \begin{array}{ccc}
       0 & 0 & 1 \\
       0 & -1 & 0 \\
       -1 & 0 & 0 \\
    \end{array}
  \right)$  \\
  $[e_{2},e_{3}]_{\alpha_{21}}=-e_{3}$&&$\hskip-0.3cm[e_{2},e_{3}]_{\alpha_{22}}=e_{3}$&\\\hline
  $[e_{1},e_{2}]_{\alpha_{23}}=-e_{1}$& &$[e_{1},e_{2}]_{\alpha_{24}}=-e_{1}$&\\
  $[e_{1},e_{3}]_{\alpha_{23}}=-e_{2}$&$\alpha_{23}=\left(
    \begin{array}{ccc}
       0 & 0  & -1 \\
       0 & 1 & 0 \\
       -1 & 0  & 0 \\
    \end{array}
  \right)$  & $\hskip-0.3cm[e_{1},e_{3}]_{\alpha_{24}}=e_{2}$ & $\alpha_{24}=\left(
    \begin{array}{ccc}
       0 & 0 & -1 \\
       0 & -1 & 0 \\
       1 & 0 & 0 \\
    \end{array}
  \right)$ \\
  $\hskip-0.3cm[e_{2},e_{3}]_{\alpha_{23}}=e_{3}$&&$[e_{2},e_{3}]_{\alpha_{24}}=-e_{3}$&\\
  \hline
\end{tabular}$$
\end{exa}
\begin{rem}
The morphisms algebras of $sl_{2}$ viewed as a color Lie algebra are all automorphisms.
\end{rem}
\begin{exa}\label{class2} The graded analogue of the Lie algebra of the group of plane of motions is defined as a complex algebra with three generators $e_{1},e_{2}$ and $e_{3}$ satisfying the commutation relations $e_{1}e_{2}+e_{2}e_{1}=e_{3},~~e_{1}e_{3}+e_{3}e_{1}=e_{2},~~e_{2}e_{3}+e_{3}e_{2}=0.$ The linear space $\mathcal{A}$ spanned by $e_{1},e_{2},e_{3}$ can be made into a $\Gamma$-graded $\varepsilon$-Lie algebra. The grading group $\Gamma$ and the commutation factor $\varepsilon$ are the same as in Example $\ref{class1}$. But the multiplication $[.,.]$ is different and is defined by
\begin{eqnarray*}
&&[e_{1},e_{1}]=e_{1}e_{1}-e_{1}e_{1}=0,~~ [e_{1},e_{2}]=e_{1}e_{2}+e_{2}e_{1}=e_{3},\\
&&[e_{2},e_{2}]=e_{2}e_{2}-e_{2}e_{2}=0,~~ [e_{1},e_{3}]=e_{1}e_{3}+e_{3}e_{1}=e_{2},\\
&&[e_{3},e_{3}]=e_{3}e_{3}-e_{3}e_{3}=0,~~ [e_{2},e_{3}]=e_{2}e_{3}+e_{3}e_{2}=0.
\end{eqnarray*}
By using   \cite[Theorem 3.14]{yuan2010hom}, we provide the following color Lie  algebras with a Hom structure. We consider an even  linear  map $\alpha:\mathcal{A}\longrightarrow \mathcal{A}$ checking $\alpha[x,y]=[\alpha(x),\alpha(y)]$ for all $x,y \in \mathcal{H}(\mathcal{A})$ in such way $(\mathcal{A},[.,.]_{\alpha}=\alpha\circ[.,.],\alpha)$ is a color Hom-Lie algebra. Then, in Table $2$, we obtain the following color Hom-Lie algebras :

\begin{center}{Table 2}\end{center}
$$\begin{tabular}{|c|c|c|c|}
  \hline
  $\hskip-0.05cm[e_{1},e_{2}]_{\alpha_{1}}=-b_{2} e_{2}-b_{1} e_{3}$ & &$[e_{1},e_{2}]_{\alpha_{2}}=b_{2} e_{2}+b_{1} e_{3}$&\\
  $\hskip-0.5cm[e_{1},e_{3}]_{\alpha_{1}}=b_{1} e_{2}+b_{2} e_{3}$ & $\alpha_{1}=\left(
    \begin{array}{ccc}
       -1 & 0 & 0 \\
      0 & b_{1} & -b_{2} \\
      0 & b_{2} & -b_{1} \\
    \end{array}
\right)$ & $[e_{1},e_{3}]_{\alpha_{2}}=b_{1} e_{2}+b_{2} e_{3}$&$\alpha_{2}=\left(
    \begin{array}{ccc}
      1 & 0 & 0 \\
      0 & b_{1} & b_{2} \\
      0 & b_{2} & b_{1} \\
    \end{array}
  \right)$ \\
   $\hskip-2cm[e_{2},e_{3}]_{\alpha_{1}}=0$& &$\hskip-1.5cm[e_{2},e_{3}]_{\alpha_{2}}=0$&\\ \hline
   $\hskip-2cm[e_{1},e_{2}]_{\alpha_{3}}=0$& &$\hskip-1.5cm[e_{1},e_{2}]_{\alpha_{4}}=0$&\\
  $\hskip-2cm[e_{1},e_{3}]_{\alpha_{3}}=0$& $\alpha_{3}=\left(
    \begin{array}{ccc}
       0 & 0 & 0 \\
      a_{1} & 0 & 0 \\
      a_{2} & 0 & 0 \\
    \end{array}
\right)$ & $\hskip-1.5cm[e_{1},e_{3}]_{\alpha_{4}}=0$ & $\alpha_{4}=\left(
    \begin{array}{ccc}
      c_{1} & 0 & 0 \\
      0 & 0 & 0 \\
      0 & 0 & 0 \\
    \end{array}
  \right)$ \\
   $\hskip-0.4cm[e_{2},e_{3}]_{\alpha_{3}}=a_{1} e_{2}+a_{2} e_{3}$& &$\hskip-1cm[e_{2},e_{3}]_{\alpha_{4}}=c_{1} e_{1}$&\\ \hline
\end{tabular}$$
where $a_{i}, b_{i},c_{i} \in \mathbb{C},~~i=1,2$.
\end{exa}
\begin{rem} It is easy to check that the even bilinear multiplications in Examples $\ref{class1}$ and $\ref{class2}$ satisfy the three following axioms$:$\\
\textsf{$\Gamma-$grading axiom$:$}
$$[\mathcal{A}_{\gamma},\mathcal{A}_{\eta}]\subseteq \mathcal{A}_{\gamma+\eta},~~\forall~~ \gamma,\eta \in \Gamma.$$
\textsf{$\varepsilon$-skew-symmetry condition$:$}
$$[x,y]=-\varepsilon(x,y)[y,x] .$$
\textsf{$\varepsilon$-Jacoby identity$:$}
$$\circlearrowleft_{x,y,z}\varepsilon(z,x)[x,[y,z]]=0$$
for nonzero homogeneous $x,y$ and $z$ and where $\circlearrowleft_{x,y,z}$ denotes summation over the cyclic permutation on $x,y,z$.
\end{rem}
\begin{center}\section{$\Gamma$-graded version Hartwig-Larsson-Silvestrov Theorem}\end{center}
In this section, we describe and prove the Hartwig-Larsson-Silvestrov Theorem \cite[Theorem 5]{hartwig2006deformations} for nongraded algebras, in the case of $\Gamma$-graded algebra. A $\mathbb{Z}_2$-graded version of this theorem was given in \cite{ammar2010hom}. We aim to consider it deeply for color algebras case.\\
Let $\mathcal{A}= \bigoplus_{\gamma \in \Gamma}\mathcal{A}_{\gamma}$
be an associative color algebra. We assume that $\mathcal{A}$
is color commutative, that is for homogeneous elements $x , y $ in $\mathcal{A}$, the identity $xy = \varepsilon(x,y) yx $ holds. Let $\sigma: \mathcal{A} \longrightarrow \mathcal{A}$ be an even color algebra endomorphism of $\mathcal{A}$. Then $\mathcal{A}$ is color bimodule over itself, the left
$($resp. right$)$ action is defined by $x\cdot_{l} y = \sigma(x)y$ $($resp. $y\cdot_{r}x = yx)$. For simplicity, we denote the module multiplication by a dot and the color multiplication by juxtaposition. In the sequel, the elements of $\mathcal{A}$ are supposed to be homogeneous.
\begin{df} Let $d \in \Gamma$. A color $\sigma$-derivation $\Delta_{d}$ on $\mathcal{A}$ is an endomorphism satisfying
\begin{eqnarray*}
(CD_{1}):~~ \Delta_{d}(\mathcal{A}_{\gamma})&\subseteq & \mathcal{A}_{\gamma+d},\\
(CD_{2}):~~ \Delta_{d}(xy)&=& \Delta_{d}(x)y+ \varepsilon(d,x) \sigma(x)\Delta_{d}(y) ,~~\forall~~x , y \in \mathcal{H(G)}.
\end{eqnarray*}
In particular for $d=0$, we have $\Delta(xy)= \Delta(x)y+ \sigma(x)\Delta(y)$, then $\Delta$ is called even color $\sigma$-derivation.
The set of all color $\sigma$-derivation is denoted by
$Der_{\sigma}^{\varepsilon}(\mathcal{A})= \bigoplus \sum_{\gamma \in \Gamma}{Der_{\gamma}}^{\varepsilon}(\mathcal{A})$.\\
The structure of $\mathcal{A}$-color module of $Der_{\sigma}(\mathcal{A})$ is as usual. Let $\Delta \in Der_{\sigma}(\mathcal{A})$, the annihilator $Ann(\Delta)$ is the set of all $x \in \mathcal{H(G)}$ such that $x\cdot \Delta=0$. We set $\mathcal{A}\cdot \Delta= \{ x\cdot \Delta~:~x \in \mathcal{H(G)} \}$ to be a color $\mathcal{A}$-module of $Der_{\sigma}(\mathcal{A})$.\\
Let $\sigma: \mathcal{A} \longrightarrow \mathcal{A}$ be a fixed endomorphism, $\Delta$ an even color $\sigma$-derivation $\Delta \in Der_{\sigma}^{\varepsilon}(\mathcal{A})$ and $\delta $ be an element in $\mathcal{A}$. Then
\end{df}
\begin{thm} If
\begin{equation}\label{ABC}
    \sigma(Ann(\Delta))\subseteq Ann(\Delta)
\end{equation}
holds, the map $[.,.]_{\sigma}:\mathcal{A}\cdot\Delta\times\mathcal{A}\cdot\Delta\longrightarrow\mathcal{A}\cdot\Delta$ defined by setting$:$
\begin{equation}\label{CDE}
    [x\cdot \Delta,y\cdot \Delta]_{\sigma}= (\sigma(x)\cdot \Delta)\circ(y\cdot \Delta)- \varepsilon(x,y)(\sigma(y)\cdot \Delta)\circ(x\cdot \Delta)
\end{equation}
where $\circ$ denotes the composition of functions, is a well-defined color algebra bracket on the $\Gamma$-graded space $\mathcal{A}\cdot \Delta$ and satisfies the following identities for $x, y \in \mathcal{H(G)}$
\begin{eqnarray}\label{principe}
    [x\cdot\Delta,y\cdot\Delta]_{\sigma}&=&(\sigma(x)\Delta(y)- \varepsilon(x,y)\sigma(y)\Delta(x))\cdot\Delta ,~~\forall~~ x, y \in \mathcal{H(G)}.
\end{eqnarray}
\begin{eqnarray}\label{FGH}
    [x\cdot\Delta,y\cdot\Delta]_{\sigma}&=& -\varepsilon(x,y)[y\cdot\Delta,x\cdot\Delta]_{\sigma} ,~~\forall~~ x , y \in \mathcal{A}.
\end{eqnarray}
In addition, if
\begin{equation}\label{IJKL}
    \Delta(\sigma(x))= \delta\sigma(\Delta(x)) ,~~\forall~~ x \in \mathcal{H(G)}
\end{equation}
holds, then\\
\begin{equation}\label{MNOP}
\circlearrowleft_{x,y,z}\varepsilon(z,x)\Big([\sigma(x)\cdot\Delta,[y\cdot\Delta,z\cdot\Delta]_{\sigma}]_{\sigma}
+\delta[x\cdot\Delta,[y\cdot\Delta,z\cdot\Delta]_{\sigma}]_{\sigma}\Big)=0
,~~\forall~~ x , y , z \in \mathcal{H}(\mathcal{A})
\end{equation}
for all $x , y$ and $z$ in $\mathcal{H(G)}$.
\end{thm}
\begin{proof} We must first show that $[.,.]_{\sigma}$ is a well-defined function. That is,\\ if $x_{1}\cdot\Delta=x_{2}\cdot\Delta$, then
\begin{equation}\label{sigma1}
    [x_{1}\cdot\Delta,y\cdot\Delta]_{\sigma}=[x_{2}\cdot\Delta,y\cdot\Delta]_{\sigma}
\end{equation}
and
\begin{equation}\label{sigma2}
    [y\cdot\Delta,x_{1}\cdot\Delta]_{\sigma}=[y\cdot\Delta,x\cdot\Delta]_{\sigma}
\end{equation}
for $x_{1}, x_{2},y \in \mathcal{H}(\mathcal{A})$. Now  $x_{1}\cdot\Delta=x_{2}\cdot\Delta$ is equivalent to  $(x_{1}-x_{2}) \in Ann(\Delta)$. Therefore, using the assumption $(\ref{ABC})$, we also have $\sigma(x_{1}-x_{2})\in Ann(\Delta)$. Then since $|x_{1}|=|x_{2}|$ and  $\sigma(x_{1}-x_{2})\in Ann(\Delta)$, we obtain$:$
\begin{eqnarray*}
 [x_{1}\cdot\Delta,y\cdot\Delta]_{\sigma}-[x_{2}\cdot\Delta,y\cdot\Delta]_{\sigma}  &=&(\sigma(x_{1})\cdot\Delta)\circ(y\cdot\Delta)-\varepsilon(x_{1},y)(\sigma(y)\cdot\Delta)\circ(x_{1}\cdot\Delta)  \\
   &=& -(\sigma(x_{2})\cdot\Delta)\circ(y\cdot\Delta)+\varepsilon(x_{2},y)(\sigma(y)\cdot\Delta)\circ(x_{2}\cdot\Delta) \\
   &=& (\sigma(x_{1}-x_{2})\cdot\Delta)\circ(y\cdot\Delta)-\varepsilon(x_{1},y)(\sigma(y)\cdot\Delta)\\
   &\circ&((x_{1}-\varepsilon(y,x_{1})
\varepsilon(x_{2},y)x_{2})\cdot\Delta).
\end{eqnarray*}
Which shows $(\ref{sigma1})$. The proof of $(\ref{sigma2})$ is analogous.\\
Next we prove $(\ref{ABC})$, which also shows that $\mathcal{A}\cdot\Delta$ is closed under $[.,.]_{\sigma}$.\\
Let $x,y,z \in \mathcal{H(G)}$ be arbitrary. Then, since $\Delta$ is an even color $\sigma$-derivation on $\mathcal{A}$ we have$:$
\begin{align*}
  [x\cdot\Delta,y\cdot\Delta]_{\sigma}(z) &=(\sigma(x)\cdot\Delta)(y\cdot\Delta)(z)- \varepsilon(x,y)(\sigma(y)\cdot\Delta)(x\cdot\Delta)(z)  \\
   &=\sigma(x)\cdot\Delta(y\Delta(z)) - \varepsilon(x,y)\sigma(y)\Delta(x\Delta(z))  \\
   &= \sigma(x)\Big(\Delta(y)\Delta(z)+ \sigma(y)\Delta^{2}(z)\Big)-\varepsilon(x,y)\sigma(y)\Big(\Delta(x)\Delta(z)+\sigma(x)\Delta^{2}(z)\Big) \\
   &= \Big(\sigma(x)\Delta(y)-\varepsilon(x,y)\sigma(y)(\Delta(x)\Big)\Delta(z)+ \Big(\sigma(x)\sigma(y) - \varepsilon(x,y)\sigma(y)\sigma(x)\Big)\Delta^{2}(z).
\end{align*}
Since $\mathcal{A}$ is color commutative, the last term is zero. Thus $(\ref{ABC})$ is true. The $\varepsilon$-skew-symmetry condition $(\ref{FGH})$ is clear from the Definition $\ref{CDE}$. Using the linearity of $\sigma$ and $\Delta$ and the definition of $[.,.]_{\sigma}$ on the formula $(\ref{ABC})$, it is also easy to see that $[.,.]_{\sigma}$ is an even bilinear map.\\
It remains to prove $(\ref{MNOP})$. Using $(\ref{ABC})$ and that $\Delta$ is an even color $\sigma$-derivation on $\mathcal{A}$, we get
\begin{eqnarray}\label{opp1}
  \varepsilon(z,x)[\sigma(x)\cdot\Delta,[y\cdot\Delta,z\cdot\Delta]_{\sigma}]_{\sigma}\nonumber&=& \varepsilon(z,x)[\sigma(x)\cdot\Delta,(\sigma(y)\cdot\Delta(z)- \varepsilon(y,z)\sigma(z)\cdot\Delta(y))\cdot\Delta]_{\sigma}\nonumber\\
   &=& \varepsilon(z,x)\Big(\sigma^{2}(x)\Delta(\sigma(y)\Delta(z)- \varepsilon(y,z)\sigma^{2}(x)\sigma(z)\Delta(y) \nonumber\\
   &+&\sigma(\varepsilon(y,z) \varepsilon(x,z+y)\sigma(z)\Delta(y)-\varepsilon(x,z+y)\sigma(y)\Delta(z))\Delta(\sigma(x)) \Big)\cdot\Delta \nonumber \\
   &=& \varepsilon(z,x) \Big(\sigma^{2}(x) \Delta(\sigma(y)\Delta(z)+\sigma^{2}(x)\sigma^{2}(y)\Delta^{2}(z) \nonumber \\
   &-& \varepsilon(y,z)\Delta(\sigma(z))\Delta(y)- \varepsilon(y,z) \sigma^{2}(z)\Delta^{2}(y))\nonumber\\
&-&( \varepsilon(x,y)\sigma^{2}(y)\sigma(\Delta(z))-\varepsilon(y,z)\varepsilon(x,z+y)\sigma^{2}(z)\sigma(\Delta(y)))\Delta(\sigma(x)\Big)\cdot\Delta,
\end{eqnarray}
where $\sigma^{2}=\sigma\circ\sigma$ and $\Delta^{2}=\Delta \circ \Delta$. Applying cyclic summation to the second and fourth term in $(\ref{opp1})$ and since  $\mathcal{A}$ is color commutative, we get
\begin{align*}&\circlearrowleft_{x,y,z}\varepsilon(z,x)\Big((\sigma^{2}(x) \Delta(\sigma(y)\Delta^{2}(z)- \varepsilon(y,z)\sigma^{2}(x)\sigma(z)\Delta^{2}(y))\cdot\Delta \Big)\\
&=\circlearrowleft_{x,y,z}\varepsilon(z,x)\Big( (\sigma^{2}(x) \Delta(\sigma(y)\Delta^{2}(z)-\varepsilon(x,y)\sigma^{2}(y)\sigma^{2}(x)\Delta^{2}(z))\cdot\Delta \Big)\\
&=0.\end{align*}
Similarly, if we apply cyclic summation to the fifth and sixth term in $(\ref{opp1})$ and use the relation $(\ref{IJKL})$ we obtain$:$
\begin{align*}
   & \circlearrowleft_{x,y,z}\varepsilon(z,x)\Big(-\varepsilon(x,y)\sigma^{2}(y)\sigma(\Delta(z))\Delta(\sigma(x))
   +\varepsilon(y,z) \varepsilon(x,y+z)\sigma^{2}(z)\Delta^{2}(y)\Delta(\sigma(x))\Big) \\
   &=\circlearrowleft_{x,y,z}\varepsilon(z,x)\Big(-\varepsilon(x,y)\sigma^{2}(y)\sigma(\Delta(z))\delta\sigma(\Delta(x))
   +\varepsilon(y,z) \varepsilon(x,y+z)\sigma^{2}(z)\Delta^{2}(y)\delta\sigma(\Delta(x))\Big)\\
   &=\delta \Big(\circlearrowleft_{x,y,z}\varepsilon(z,x)((-\varepsilon(x,y)\sigma^{2}(y)\sigma(\Delta(z))\sigma(\Delta(x))+ \varepsilon(x,y)\sigma^{2}(y)\sigma(\Delta(x))\sigma(\Delta(z)))\cdot\Delta )\Big) \\
   &= 0,
\end{align*}
where we again use the color commutativity of $\mathcal{A}$. Consequently, the only terms in the right hand side of $(\ref{opp1})$ which do not vanish when take cyclic summation are
\begin{align}\label{opp2}&\circlearrowleft_{x,y,z}\varepsilon(z,x)[\sigma(x)\cdot\Delta,[y\cdot\Delta,z\cdot\Delta]_{\sigma}]_{\sigma}\nonumber\\
    &=\circlearrowleft_{x,y,z}\varepsilon(z,x)\Big((\sigma^{2}(x) \Delta(\sigma(y))\Delta(z)-\varepsilon(y,z)\sigma^{2}(x)\Delta(\sigma(z)\Delta(y))\cdot\Delta \Big).
\end{align}
We now consider the other term in $(\ref{MNOP})$. First that from $(\ref{IJKL})$ we have
$$[y\cdot\Delta, z\cdot\Delta]_{\sigma}=(\sigma(y)\Delta(z)- \varepsilon(y,z)\sigma(z)\Delta(y))\cdot \Delta. $$
Using first this and then $(\ref{principe})$, we get
\begin{eqnarray*}
 &&\ \  \delta[x\cdot\Delta,[y\cdot\Delta,z\cdot\Delta]_{\sigma}]_{\sigma} =\delta[x\cdot\Delta, (\sigma(y)\Delta(z)
  - \varepsilon(y,z)\sigma(z)\Delta(y))\cdot\Delta ]_{\sigma}  \\
  &&\ \ = \delta\Big(\sigma(x)\Delta(\sigma(y)\Delta(z)-\varepsilon(y,z)\sigma(z)\Delta(y))
    - \varepsilon(x,y+z)\sigma(\sigma(y)\Delta(z)-\varepsilon(y,z)\sigma(z)\Delta(y)))\cdot\Delta(x)\Big) \\
   &&\ \ =\delta\Big((\varepsilon(y,z)\sigma(x)\Delta^{2}(z)\sigma(y)+\varepsilon(y,z)\sigma(x)\sigma(\Delta(z))\Delta(\sigma(y))  \\
   &&\ \  -\sigma(x)\Delta^{2}(y)\sigma(z)-\sigma(x)\sigma(\Delta(y))\Delta(\sigma(z))-\varepsilon(x,y)\varepsilon(y,z)
    - \sigma(\Delta(z))\sigma^{2}(y)\Delta(x)+\sigma(\Delta(y))\sigma^{2}(z)\Delta(x))\cdot \Delta\Big) \\
   &&\ \  = \varepsilon(y,z)\delta\sigma(x)\Delta^{2}(z)\sigma(y)+\varepsilon(y,z)\delta\sigma(x)\sigma(\Delta(z))\Delta(\sigma(y))\\
   &&\ \  -\delta\sigma(x)\Delta^{2}(y)\sigma(z)-\delta\sigma(x)\sigma(\Delta(y))\Delta(\sigma(z))-\varepsilon(x,y)\varepsilon(y,z)
    -\delta\sigma(\Delta(z))\sigma^{2}(y)\Delta(x)+\delta\sigma(\Delta(y))\sigma^{2}(z)\Delta(x))\cdot \Delta.
\end{eqnarray*}
Using $(\ref{IJKL})$, this equal to
\begin{eqnarray*}
   &&\Big( \varepsilon(y,z)\delta\sigma(x)\Delta^{2}(z)\sigma(y)+\varepsilon(y,z)\sigma(x)\Delta(\sigma(z))\Delta(\sigma(y))
   - \delta\sigma(x)\Delta^{2}(y)\sigma(z)-\delta\sigma(x)\sigma(\Delta(y))\Delta(\sigma(z))\\
   &&\ \ -\varepsilon(x,y)\varepsilon(y,z)\Delta(\sigma(z))\sigma^{2}(y)\Delta(x)+\Delta(\sigma(y))\sigma^{2}(z)\Delta(x)\Big)\cdot \Delta \\
   &&\ \ = \Big(\varepsilon(y,z)\delta\sigma(x)\Delta^{2}(z)\sigma(y)-\delta\sigma(x)\Delta^{2}(y)\sigma(z) -\varepsilon(x,y)\varepsilon(y,z)
   \Delta(\sigma(z))\sigma^{2}(y)\Delta(x) +\Delta(\sigma(y))\sigma^{2}(z)\Delta(x)\Big)\cdot\Delta.
\end{eqnarray*}
The first two terms of this last expression vanish after a cyclic summation, so we get
\begin{align}\label{opp3}&\circlearrowleft_{x,y,z}\varepsilon(z,x)\delta[x\cdot\Delta,[y\cdot\Delta,z\cdot\Delta]_{\sigma}]_{\sigma}\nonumber\\&=
    \circlearrowleft_{x,y,z}\varepsilon(z,x)\Big((-\varepsilon(x,y)\varepsilon(y,z)\Delta(\sigma(z))\sigma^{2}(y)\Delta(x)+ \Delta(\sigma(y))\sigma^{2}(z)\Delta(x))\cdot\Delta \Big).
\end{align}
Finally, combining this with $(\ref{opp2})$ we deduce
\begin{align*}&\circlearrowleft_{x,y,z}\varepsilon(z,x)([\sigma(x)\cdot\Delta,[y\cdot\Delta,z\cdot\Delta]_{\sigma}]_{\sigma}
+\delta[x\cdot\Delta,[y\cdot\Delta,z\cdot\Delta]_{\sigma}]_{\sigma})\\
&=\circlearrowleft_{x,y,z}\varepsilon(z,x)[\sigma(x)\cdot\Delta,[y\cdot\Delta,z\cdot\Delta]_{\sigma}]_{\sigma}
+\circlearrowleft_{x,y,z}\varepsilon(z,x)\delta[x\cdot\Delta,[y\cdot\Delta,z\cdot\Delta]_{\sigma}]_{\sigma}\\
&=\circlearrowleft_{x,y,z}\varepsilon(z,x)\Big((\sigma^{2}(x) \Delta(\sigma(y))\Delta(z)-\varepsilon(y,z)\sigma^{2}(x)\Delta(\sigma(z)\Delta(y))\cdot\Delta \Big)\\
&+\circlearrowleft_{x,y,z}\varepsilon(z,x)\Big((-\varepsilon(x,y)\varepsilon(y,z)\Delta(\sigma(y))\sigma^{2}(x)\Delta(z)+ \Delta(\sigma(z))\sigma^{2}(x)\Delta(y))\cdot\Delta \Big)\\
&=0,\end{align*}
as was to be shown. The proof is complete.
\end{proof}
\begin{center}\section{Cohomology and Representations of color Hom-Lie algebras}\end{center}
In this section we define a family of cohomology complexes of color Hom-Lie algebras, discuss the representations in connection with cohomology   and provide an example of computation.
\subsection{Cohomology of color Hom-Lie algebras}
We extend to color Lie algebras, the concept of $\A$-module  introduced in \cite{sadaoui2011cohomology,benayadi2010hom,sheng2010representations}, and then define a family of cohomology complexes for colr Hom-Lie algebras.\\

Let $(\mathcal{A},[.,.],\varepsilon,\alpha)$ be a color Hom-Lie algebra, $(M,\beta)$ be a pair of $\Gamma$-graded vector space $M$ and an even homomorphism of vectors spaces $\beta:M \longrightarrow M$, and
$$\begin{array}{cccc}
                                                [.,.]_{M}: & \mathcal{A} \times M  & \longrightarrow  & M \\
                                                                   & (g,m) & \longmapsto  & [g,m]_{M}
                                              \end{array}$$
 be an even bilinear map satisfying $[\mathcal{A}_{\gamma_{1}},M_{\gamma_{2}}]\subseteq M_{\gamma_{1}+\gamma_{2}} $ for all $\gamma_{1},\gamma_{2} \in \Gamma$.
\begin{df} The triple $(M,[.,.]_{M},\beta)$ is said to be $\mathcal{A}$-module if the even bilinear map $[.,.]_{M}$ satisfies
\begin{equation}\label{module1}
   \beta([x,m]_{M})= [\alpha(x),\beta(m)]_{M}
\end{equation}
and
\begin{equation}\label{module2}
    [[x,y],\beta(m)]_{M}=[\alpha(x),[y,m]]_{M}-\varepsilon(x,y)[\alpha(y),[x,m]]_{M}.
\end{equation}
\end{df}

The cohomology of color Lie algebras was introduced in \cite{scheunert1979generalized}. In the following, we  define cohomology complexes of color Hom-Lie algebras.\\
The set $C^{n}(\mathcal{A},M)$ of $n$-cochains on space $\mathcal{A}$ with values in $M$, is the set of $n$-linear maps $f:\mathcal{A}^{n}\longrightarrow M$ satisfying
$$f(x_{1},...,x_{i},x_{i+1},...,x_{n})=-\varepsilon(x_{i},x_{i+1})f(x_{1},...,x_{i+1},x_{i},...,x_{n}),~~\forall~~ 1\leq i \leq n-1.$$
For $n=0$, we have $C^{0}(\mathcal{A},M)=M$.\\
The map $f$ is called even (resp. of degree $\gamma$) when $f(x_{1},...,x_{i},...,x_{n}) \in M_{0}$ for all elements
$(x_{1},...,x_{n}) \in \mathcal{A}^{\otimes n}$ (resp. $f(x_{1},...,x_{i},...,x_{n}) \in M_{\gamma}$ for all elements $(x_{1},...,x_{n}) \in \mathcal{A}^{\otimes n}$ of degree $\gamma$).\\
A  $n$-cochain on $\mathcal{A}$ with values in $M$ is defined  to be a $n$-cochain $f \in C^{n}(\mathcal{A},M)$ such that it is compatible with $\alpha$ and $\beta$ in the sense that $f\circ \alpha=\beta \circ f$.\\
Denote by $C_{\alpha,\beta}^{n}(\mathcal{A},M)$ the set of $n$-cochains$:$
\begin{equation}\label{module3}
    C_{\alpha,\beta}^{n}(\mathcal{A},M)=\{f:\mathcal{A}\longrightarrow M ~:~ f\circ \alpha=\beta \circ f \}.
\end{equation}
We extend this definition to the case of integers $n < 0$  and set
$$C_{\alpha,\beta}^{n}(\mathcal{A},M)=\{0\}\ \text{ if } n < -1\quad \text{ and }\quad
C^{0}(\mathcal{A},M)=M.$$
A homogeneous element $f \in C_{\alpha,\beta}^{n}(\mathcal{A},M)$ is called $n$-cochain.\\
Next, for a given  integer $r$, we define the coboundary operator $\delta_{r}^{n}$.
\begin{df} We call, for $n\geq 1$ and for a any integer $m$, a $n$-coboundary operator of the color Hom-Lie algebra $(\mathcal{A},[.,.],\varepsilon,\alpha)$ the linear map $\delta_{r}^{n}:C_{\alpha,\beta}^{n}(\mathcal{A},M)\longrightarrow C_{\alpha,\beta}^{n+1}(\mathcal{A},M)$ defined by
\begin{eqnarray}\label{ECHLC}
&& \delta_{r}^{n}(f)(x_{0},....,x_{n}) = \\
&&\quad  \nonumber \sum \limits_{0\leq s < t \leq n}(-1)^{t}\varepsilon(x_{s+1}+...+x_{t-1},x_{t})
 f(\alpha(x_{0}),...,\alpha(x_{s-1}),[x_{s},x_{t}],\alpha(x_{s+1}),...,\widehat{x_{t}},...,\alpha(x_{n}))\\
&& \nonumber \quad + \sum \limits_{s=0}^{n}(-1)^{s}\varepsilon(\gamma+x_{0}+...+x_{s-1},x_{s})[\alpha^{r+n-1}(x_{s}),f(x_{0},...,\widehat{x_{s}},..,x_{n})]_{M},
\end{eqnarray}
where $f \in C_{\alpha,\beta}^{n}(\mathcal{A},M)$, $\gamma$ is the degree of $f$, $(x_{0}, ...., x_{n}) \in \mathcal{H}(\mathcal{A})^{\otimes n} $ and  $\widehat{x}$ indicates that the element $x$ is omitted.\\

In the sequel we assume that the color Hom-Lie algebra $(\mathcal{A},[.,.],\varepsilon,\alpha)$ is multiplicative.
\end{df}
For $n=1$, we have
$$\begin{array}{cccc}
  \delta_{r}^{1}: & C^{1}(\mathcal{A},M) & \longrightarrow & C^{2}(\mathcal{A},M)  \\
   & f & \longmapsto & \delta_{r}^{1}(f)
\end{array}$$
such that for two homogeneous elements $x,y$ in $\mathcal{A}$
\begin{equation}\label{cobord1}
    \delta_{r}^{1}(f)(x,y) =\varepsilon(\gamma,x)\rho(x)(f(y))-\varepsilon(\gamma+x,y)\rho(y)(f(x))-f([x,y])
\end{equation}
and for $n=2$, we have
$$\begin{array}{cccc}
  \delta_{r}^{2}: & C^{2}(\mathcal{A},M) & \longrightarrow & C^{3}(\mathcal{A},M)  \\
   & f & \longmapsto & \delta_{r}^{2}(f)
\end{array}$$
such that, for three homogeneous elements $x,y,z$ in $\mathcal{A}$, we have
\begin{eqnarray}\label{2-cocycle}
\delta_{r}^2(f)(x,y,z)&=&\varepsilon(\gamma,x)\rho(\alpha(x))(f(y,z))-\varepsilon(\gamma+x,y)\rho(\alpha(y))(f(x,z))\nonumber\\
&+&\varepsilon(\gamma+x+y,z)\rho(\alpha(z))(f(x,y))-f([x,y],\alpha(z))\nonumber\\
&+&\varepsilon(y,z)f([x,z],\alpha(y))+f(\alpha(x),[y,z]).
\end{eqnarray}

\begin{lem} With the above notations, for any $f\in C^{n}_{\alpha,\beta}(\mathcal{A},M)$, we have
$$\delta_{r}^{n}(f)\circ \alpha=\beta \circ \delta_{r}^{n}(f),~~\forall~~ n \geq 2.$$
Thus, we obtain a well defined map $\delta_{r}^{n}:C_{\alpha,\beta}^{n}(\mathcal{A},M)\longrightarrow C_{\alpha,\beta}^{n+1}(\mathcal{A},M).$
\end{lem}
\begin{proof} Let $f\in C^{n}_{\alpha,\beta}(\mathcal{A},M)$ and $(x_{0},....,x_{n})\in \mathcal{H}(\mathcal{A}^{\otimes n+1})$, we have
\begin{eqnarray*}
&&\delta_{r}^{n}(f)\circ \alpha(x_{0},....,x_{n})\\
&&=\small \small \delta_{r}^{n}(f)(\alpha(x_{0}),....,\alpha(x_{n}))\\
&&=\small \sum \limits_{0\leq s < t \leq n}(-1)^{t}\varepsilon(x_{0}+...+x_{t-1},x_{t})f(\alpha^{2}(x_{0}),...,
\alpha^{2}(x_{s-1}), [\alpha(x_{s}),\alpha(x_{t})],\alpha^{2}(x_{s+1}),...,\widehat{x_{t}},...,\alpha^{2}(x_{n}))\\
&&+\sum \limits_{s=0}^{n}(-1)^{s}\varepsilon(\gamma+x_{0}+...+x_{s-1},x_{s})
[\alpha^{n+r}(x_{s}),f(\alpha(x_{0}),...,\widehat{x_{s}},..,\alpha(x_{n}))]_{M}\\
&&=\sum \limits_{0\leq s < t \leq n}(-1)^{t}\varepsilon(x_{0}+...+x_{t-1},x_{t})f\circ \alpha(\alpha(x_{0}),...,\alpha(x_{s-1}), [x_{s},x_{t}], \alpha(x_{s+1}),...,\widehat{x_{t}},...,\alpha(x_{n}))\\
&&+\sum \limits_{s=0}^{n}(-1)^{s}\varepsilon(\gamma+x_{0}+...+x_{s-1},x_{s}) [\alpha^{n+r}(x_{s}), \beta \circ f(x_{0},...,\widehat{x_{s}},..,x_{n})]_{M}\\
&&=\sum \limits_{0\leq s < t \leq n}(-1)^{t}\varepsilon(x_{0}+...+x_{t-1},x_{t})\beta \circ f(\alpha(x_{0}),...,\alpha(x_{s-1}),[x_{s},x_{t}], \alpha(x_{s+1}),...,\widehat{x_{t}},...,\alpha(x_{n}))\\
&&+\sum \limits_{s=0}^{n}(-1)^{s}\varepsilon(\gamma+x_{0}+...+x_{s-1},x_{s})\beta([\alpha^{n+r-1}(x_{s}),
f(x_{0},...,\widehat{x_{s}},..,x_{n})]_{M})\\
&&=\beta \circ \delta_{r}^{n}(f)(x_{0},....,x_{n}).
\end{eqnarray*}
Then $\delta_{r}^{n}(f)\circ \alpha=\beta \circ \delta_{r}^{n}(f)$ which completes the proof.
\end{proof}
\begin{thm}\label{ker} Let $(\mathcal{A},[.,.],\varepsilon,\alpha)$ be a color Hom-Lie algebra and $(M,\beta)$ be an $\mathcal{A}$-module.
Then the pair $(\bigoplus_{n\geq 0}C_{\alpha,\beta}^{n}, \delta_{r}^{n})$ is a cohomology complex. That is  the maps $\delta_{r}^{n}$ satisfy  $\delta_{r}^{n}\circ \delta_{r}^{n-1}=0,~~\forall~~ n \geq 2,\forall~~ r \geq 1$
\end{thm}
 \begin{proof} For any $f \in C^{n-1}(\mathcal{A},M)$, we have
 \begin{eqnarray}\label{1}
  && \delta_{r}^{n}\circ \delta_{r}^{n-1}(f)(x_{0},....,x_{n})=\nonumber\\
&& \quad\quad \sum_{s < t}(-1)^{t}\varepsilon(x_{0}+...+x_{t-1},x_{t})
 \delta_{r}^{n+1}(f)(\alpha(x_{0}),...,\alpha(x_{s-1}),[x_{s},x_{t}],\alpha(x_{s+1}),...,
\widehat{x_{t}},...,\alpha(x_{n}))
\end{eqnarray}
\begin{eqnarray}\label{2}
     &+&\sum_{s=0}^{n}(-1)^{s}\varepsilon(f+x_{0}+...+x_{s-1},x_{s})[\alpha^{r+n-1}(x_{s}),
 \delta_{r}^{n-1}(f)(x_{0},...,\widehat{x_{s}},..,x_{n})]_{M}.
\end{eqnarray}
From $(\ref{1})$ we have
\begin{eqnarray*}
&&\delta_{r}^{n-1}(f)(\alpha(x_{0}),...,\alpha(x_{s-1}),[x_{s},x_{t}],\alpha(x_{s+1}),...,
\widehat{x_{t}},...,\alpha(x_{n}))\\
&=&\sum \limits_{s^{'} < t^{'}< s }(-1)^{t^{'}}\varepsilon(x_{s^{'}+1}+...+x_{t^{'}-1},x_{t^{'}})f(\alpha^{2}(x_{0})
,...,\alpha^{2}(x_{s^{'}-1}),[\alpha(x_{s^{'}}),\alpha(x_{t^{'}})],\alpha^{2}(x_{s^{'}+1}),
\end{eqnarray*}
 \begin{equation}\label{01}
...,\widehat{x_{t^{'}}},...,\alpha^{2}(x_{s-1}),\alpha([x_{s},x_{t}]),\alpha^{2}(x_{s+1}),...,\widehat{x_{t}},...,\alpha^{2}(x_{n}))
\end{equation}
\begin{eqnarray*}
&+&\sum \limits_{s^{'}< s }(-1)^{s}\varepsilon(x_{s^{'}+1}+...+x_{s-1},x_{s})
 \end{eqnarray*}
 \begin{eqnarray}\label{02}
&& f(\alpha^{2}(x_{0}),...,\alpha^{2}(x_{s^{'}-1}),[\alpha(x_{s^{'}-1}),[x_{s},x_{t}]],\alpha^{2}(x_{s^{'}+1}),...,\widehat{x_{s,t}},...,\alpha^{2}(x_{n}))
\end{eqnarray}
\begin{eqnarray*}
&+&\sum \limits_{s^{'} < s< t^{'}< t }(-1)^{t^{'}}\varepsilon(x_{s^{'}+1}+...+[x_{s},x_{t}]+...+x_{t^{'}-1},x_{t^{'}})
 \end{eqnarray*}
 \begin{equation}\label{03}
f(\alpha^{2}(x_{0}),...,\alpha^{2}(x_{s^{'}-1}),[\alpha(x_{s^{'}}),\alpha(x_{t^{'}})],\alpha^{2}(x_{s^{'}+1}),\alpha([x_{s},x_{t}])
,...,\widehat{x_{t^{'}}},...,\alpha^{2}(x_{n}))
\end{equation}
\begin{eqnarray*}
&+&\sum \limits_{s^{'}< s< t< t^{'}}
(-1)^{t^{'}}\varepsilon(x_{s^{'}+1}+...x_{s-1}+[x_{s},x_{t}]+x_{s+1}+...+\widehat{x_{t}}+...+x_{t^{'}-1},x_{t^{'}})\end{eqnarray*}
\begin{equation}\label{04}
 f(\alpha^{2}(x_{0}),...,,\alpha^{2}(x_{s^{'}-1}),[\alpha(x_{s^{'}}),\alpha(x_{t^{'}})],\alpha^{2}(x_{s^{'}+1}),
,\alpha([x_{s},x_{t}]),...,\widehat{x_{t}},...,\widehat{x_{t^{'}}},...,\alpha^{2}(x_{n}))
\end{equation}
\begin{eqnarray}\label{05}
&+&\sum \limits_{ s< t^{'}< t}
(-1)^{t^{'}}\varepsilon(x_{s+1}+...+x_{t^{'}-1},x_{t^{'}}) f(\alpha^{2}(x_{0}),...,[[x_{s},x_{t}],\alpha(x_{t^{'}})],
\alpha^{2}(x_{s+1}),...,\widehat{x_{t,t^{'}}},...,\alpha^{2}(x_{n}))
\end{eqnarray}
\begin{eqnarray*}
&+&\sum \limits_{ s<t <t^{'}}(-1)^{t^{'}-1}\varepsilon(x_{s+1}+...+\widehat{x_{t}}+...+x_{t^{'}-1},x_{t^{'}})\end{eqnarray*}
\begin{eqnarray}\label{06}
f(\alpha^{2}(x_{0}),...,\alpha^{2}(x_{s-1}),[[x_{s},x_{t}],\alpha(x_{t^{'}})],\alpha^{2}(x_{s+1}),...,\widehat{x_{t,t^{'}}},...,\alpha^{2}(x_{n}))
\end{eqnarray}
\begin{eqnarray*}
&+&\sum \limits_{ s<s^{'}< t^{'}< t}(-1)^{t^{'}}\varepsilon(x_{s^{'}+1}+...+x_{t^{'}-1},x_{t^{'}})
f(\alpha^{2}(x_{0}),...\alpha^{2}(x_{s-1}),\alpha([x_{s},x_{t}]),\end{eqnarray*}
\begin{eqnarray}\label{07}
[\alpha(x_{s^{'}}),\alpha(x_{t^{'}})],...,\widehat{x_{t^{'}}},...,\widehat{x_{t}},...,\alpha^{2}(x_{n}))
\end{eqnarray}
\begin{eqnarray*}
&+&\sum \limits_{ s<s^{'}< t<t^{'}}(-1)^{t^{'}}\varepsilon(x_{s^{'}+1}+...+\widehat{x_{t}}+...+x_{t^{'}-1},x_{t^{'}})
\end{eqnarray*}
\begin{eqnarray}\label{08}
f(\alpha^{2}(x_{0}),...\alpha^{2}(x_{s-1}),\alpha([x_{s},x_{t}]),\alpha^{2}(x_{s+1}),..., [\alpha(x_{s^{'}}),\alpha(x_{t^{'}})],...,\widehat{x_{t}},...\widehat{x_{t^{'}}},...,\alpha^{2}(x_{n}))
\end{eqnarray}
\begin{eqnarray*}
&+&\sum_{ t<s^{'}<t^{'}}(-1)^{t^{'}}\varepsilon(x_{s^{'}+1}+...+\widehat{x_{t,t^{'}}}+...+x_{t^{'}-1},x_{t^{'}})
\end{eqnarray*}
\begin{eqnarray}\label{O9}
f(\alpha^{2}(x_{0}),...\alpha^{2}(x_{s-1}),\alpha([x_{s},x_{t}]),\alpha^{2}(x_{s+1}),...,\widehat{x_{t}},..., [\alpha(x_{s^{'}}),\alpha(x_{t^{'}})],...,\widehat{x_{t^{'}}},...,\alpha^{2}(x_{n}))
\end{eqnarray}
\begin{eqnarray}\label{010}
&+&\sum \limits_{ 0< s^{'}< s}(-1)^{s^{'}}\varepsilon(\gamma+x_{0}+...+x_{s^{'}-1},x_{s^{'}}) [\alpha^{r+n-1}(x_{s^{'}}),f(\alpha(x_{0}),...\widehat{x_{s^{'}}}
,[x_{s},x_{t}],...,\widehat{x_{t^{'}}},...,\alpha(x_{n}))]_{M}
\end{eqnarray}
\begin{eqnarray*}
&+&(-1)^{s}\varepsilon(\gamma+x_{0}+...+x_{s-1},[x_{s},x_{t}])
\end{eqnarray*}
\begin{eqnarray}\label{011}
[\alpha^{r+n-2}([x_{s},x_{t}]),
f(\alpha(x_{0}),...\widehat{[x_{s},x_{t}]},\alpha(x_{s+1})...,\widehat{x_{t}},...,\alpha(x_{n}))]_{M}
\end{eqnarray}
\begin{eqnarray*}
& +&\sum \limits_{ s< s^{'}< t}(-1)^{s^{'}}\varepsilon(\gamma+x_{0}+...+[x_{s},x_{t}]+...+x_{s^{'}-1},x_{s^{'}})\end{eqnarray*}
\begin{eqnarray}\label{012}
&&[\alpha^{r+n-2}(x_{s^{'}}),f(\alpha(x_{0}),...,[x_{s},x_{t}],...,\widehat{x_{s^{'},t}},...,\alpha(x_{n}))]_{M}
\end{eqnarray}
\begin{eqnarray*}
&+&\sum \limits_{ t< s^{'}}(-1)^{s^{'}}\varepsilon(\gamma+x_{0}+..[x_{s},x_{t}]+...+\widehat{x_{t}}+...+x_{s^{'}-1},x_{s^{'}})\end{eqnarray*}
\begin{eqnarray}\label{013}
&&[\alpha^{r+n-2}(x_{s^{'}}),f(\alpha(x_{0}),...,[x_{s},x_{t}],...,\widehat{x_{t,s^{'}}},...,\alpha(x_{n}))]_{M}.
\end{eqnarray}
The identity $(\ref{2})$ implies that
\begin{eqnarray*}[\alpha^{r+n-1}(x_{s}), \delta_{r}^{n-1}(f)(x_{0},...,\widehat{x_{s}},..,x_{n})]_{M}
&=&[\alpha^{r+n-1}(x_{s}),\sum \limits_{s^{'} < t^{'}< s }(-1)^{t^{'}}\varepsilon(x_{s^{'}+1}+...+x_{t^{'}-1},x_{t^{'}})
\end{eqnarray*}
\begin{eqnarray*}
&&f(\alpha(x_{0}),...,\alpha(x_{s^{'}-1}),[x_{s^{'}},x_{t^{'}}],\alpha(x_{s^{'}+1}),...,\widehat{x_{s^{'},t^{'},t}},\alpha(x_{s+1}),...,\alpha(x_{n}))]_{M}
\end{eqnarray*}
\begin{eqnarray*}
&+&[\alpha^{r+n-1}(x_{s}),\sum_{s^{'} < s<t }(-1)^{t^{'}-1}\varepsilon(x_{s^{'}+1}+...+\widehat{x_{s}}+...+x_{t^{'}-1},x_{t^{'}})\end{eqnarray*}
\begin{equation}\label{015}
f(\alpha(x_{0}),...,\alpha(x_{s^{'}-1}),[x_{s^{'}},x_{t^{'}}],\alpha(x_{s^{'}+1}),...,\widehat{x_{t,s^{'}}},...,\alpha(x_{n}))]_{M}
\end{equation}
\begin{eqnarray*}&+&[\alpha^{r+n-1}(x_{s}),\sum \limits_{s < s^{'}<t^{'} }(-1)^{t^{'}}\varepsilon(x_{s^{'}+1}+...+x_{t^{'}-1},x_{t^{'}})\end{eqnarray*}
\begin{equation}\label{016}
f(\alpha(x_{0}),...,\widehat{x_{s}},...,\alpha(x_{s^{'}-1}),[x_{s^{'}},x_{t^{'}}],\alpha(x_{s^{'}+1}),...,\widehat{x_{t^{'}}},...,\alpha(x_{n}))]_{M}
\end{equation}
\begin{equation}\label{017}
+[\alpha^{r+n-1}(x_{s}),\sum_{s^{'}=0}^{s-1}(-1)^{s}\varepsilon(\gamma+x_{0}+...+x_{s^{'}-1},x_{s^{'}})[\alpha^{r+n-2}(x_{s^{'}}),
f(x_{0},...,\widehat{x_{s^{'},s}},..,x_{n})]_{M}]_{M}
\end{equation}
\begin{eqnarray}
&+&[\alpha^{r+n-1}(x_{s}),\sum_{s^{'}=s+1}^{n}(-1)^{s^{'}-1}\varepsilon(\gamma+x_{0}+...+\widehat{x_{s}}....+x_{s^{'}-1},x_{s^{'}})
[\alpha^{n+r-2}(x_{s^{'}}),f(x_{0},...,\widehat{x_{s^{'},s}},..,x_{n})]_{M}]_{M}.
\nonumber \\ \label{018} &&
\end{eqnarray}
By the $\varepsilon$-Hom-Jacobi condition, we obtain
$$\sum \limits_{ s<t }(-1)^{t}\varepsilon(x_{s+1}+...+x_{t-1},x_{t})((\ref{02})+(\ref{05})+(\ref{06}))=0.$$
Using $(\ref{module2})$ and $(\ref{module3})$, we  obtain
\begin{eqnarray*}
  (\ref{011}) &=&[\alpha^{r+n-2}([x_{s},x_{t}]),f(\alpha(x_{0}),...\alpha(x_{s-1}),\widehat{\alpha([x_{s},x_{t}])},\alpha(x_{s+1})....,\widehat{x_{t}},...,\alpha(x_{n}))]_{M}  \\
   &=& [\alpha^{n+r-1}(x_{s}),[\alpha^{r+n-2}(x_{t}),f(x_{0},...,\widehat{x_{s,t}},..,x_{n})]_{M}]_{M}\\
   &-&[\alpha^{r+n-1}(x_{t}),[\alpha^{r+n-2}(x_{s}),f(x_{0},...,\widehat{x_{s,t}},..,x_{n})]_{M}]_{M}.
\end{eqnarray*}
Thus
\begin{eqnarray*}
&&\ \ \ \sum \limits_{ s < t }(-1)^{t}\varepsilon(x_{s+1}+...+x_{t-1},x_{t})(\ref{011})
+\sum \limits_{s=0}^{n}(-1)^{s}\varepsilon(\gamma+x_{0}+...+x_{s-1},x_{s})(\ref{017})\\
&&\ \ \ +\sum \limits_{s=0}^{n}(-1)^{s}\varepsilon(\gamma+x_{0}+...+x_{s-1},x_{s})(\ref{018})=0.
\end{eqnarray*}
By a simple calculation, we get
\begin{eqnarray*}
\sum \limits_{ s < t }(-1)^{t}\varepsilon(x_{s+1}+...+x_{t-1},x_{t})(\ref{010})
 +\sum \limits_{s=0}^{n}(-1)^{s}\varepsilon(\gamma+x_{0}+...+x_{s-1},x_{s})(\ref{016})&=&0,\\
 \sum \limits_{ s < t }(-1)^{t}\varepsilon(x_{s+1}+...+x_{t-1},x_{t})(\ref{013})
+\sum \limits_{s=0}^{n}(-1)^{s}\varepsilon(\gamma+x_{0}+...+x_{s-1},x_{s})(\ref{015})&=&0,
\end{eqnarray*}
and
\begin{eqnarray*}&&\ \ \ \sum \limits_{ s < t }(-1)^{t}\varepsilon(x_{s+1}+...+x_{t-1},x_{t})((\ref{03})+(\ref{08}))\\
&&\ \ \ =\sum \limits_{ s < t }(-1)^{t}\varepsilon(x_{s+1}+...+x_{t-1},x_{t})\Big(\sum \limits_{s^{'} < s< t^{'}< t} (-1)^{t^{'}} \varepsilon(x_{s^{'}+1}+...+[x_{s},x_{t}]+...+x_{t^{'}-1},x_{t^{'}})\\
&&\ \ \ \quad f(\alpha^{2}(x_{0}),..., \alpha^{2}(x_{s^{'}-1}),[\alpha(x_{s^{'}}),\alpha(x_{t^{'}})] ,\alpha^{2}(x_{s^{'}+1})
,\alpha([x_{s},x_{t}]),..., \widehat{x_{t^{'}}},...,\alpha^{2}(x_{n}))\Big)\\
&&\ \ \ +\sum \limits_{ s < t }(-1)^{t}\varepsilon(x_{s+1}+...+x_{t-1},x_{t})\Big( \sum_{ s<s^{'}< t<t^{'}}
(-1)^{t^{'}}\varepsilon(x_{s^{'}+1}+...+\widehat{x_{t}}+...+x_{t^{'}-1},x_{t^{'}})\\
&&\ \ \ \quad f(\alpha^{2}(x_{0}),...\alpha^{2}(x_{s-1}),\alpha([x_{s},x_{t}]), \alpha^{2}(x_{s+1}),...,\widehat{x_{t,t^{'}}},...[\alpha(x_{s^{'}}),\alpha(x_{t^{'}})],...,\alpha^{2}(x_{n}))\Big)\\
&&\ \ \ =0.\end{eqnarray*}
Similarly, we have
$$\sum \limits_{ s < t }(-1)^{t}\varepsilon(x_{s+1}+...+x_{t-1},x_{t})((\ref{01})+(\ref{O9}))=0$$
and
 $$\sum \limits_{ s < t }(-1)^{t}\varepsilon(x_{s+1}+...+x_{t-1},x_{t})((\ref{04})+(\ref{07}))=0.$$
Therefore  $\delta_{r}^{n}\circ \delta_{r}^{n-1}=0$. Which completes the proof.
\end{proof}

Let $Z_{r}^{n}(\mathcal{A},M)$ (resp. $B_{r}^{n}(\mathcal{A},M)$) denote the kernel of $\delta_{r}^{n}$ (resp. the image of $\delta_{r}^{n-1}$). The spaces $Z_{r}^{n}(\mathcal{A},M)$ and $B_{r}^{n}(\mathcal{A},M)$ are graded submodules of $C_{\alpha,\beta}^{n}(\mathcal{A},M)$ and according to Proposition \ref{ker}, we have
\begin{equation}\label{BZ}
    B_{r}^{n}(\mathcal{A},M)\subseteq Z_{r}^{n}(\mathcal{A},M).
\end{equation}
The elements of $Z_{r}^{n}(\mathcal{A},M)$ are called $n$-cocycles, and  the elements of $B_{r}^{n}(\mathcal{A},M)$ are called the $n$-coboundaries. Thus, we define a so-called cohomology groups
$$H_{r}^{n}(\mathcal{A},M)=\frac{Z_{r}^{n}(\mathcal{A},M)}{B_{r}^{n}(\mathcal{A},M)}.$$
We denote by $H_{r}^{n}(\mathcal{A},M)=\bigoplus_{d \in \Gamma}(H_{r}^{n}(\mathcal{A},M))_{d}$ the space of all $r$-cohomology group of degree $d$ of the color Hom-Lie algebra  $\mathcal{A}$ with values in $M$.\\
Two elements of $Z_{r}^{n}(\mathcal{A},M)$ are said to be cohomologeous if their residue class modulo $B_{r}^{n}(\mathcal{A},M)$ coincide, that is if their difference lies in $B_{r}^{n}(\mathcal{A},M)$.

\subsection{Adjoint representations of color Hom-Lie algebras}
In this section, we generalize to color Hom-Lie algebras some results from \cite{sadaoui2011cohomology} and \cite{sheng2010representations}. Let $(\mathcal{A},[.,.],\varepsilon,\alpha)$ be a multiplicative color Hom-Lie algebra. We consider $\mathcal{A}$ represents on itself via bracket with respect to the morphism $\alpha$.
\begin{df} Let $(\mathcal{A},[.,.],\varepsilon,\alpha)$ be a color Hom-Lie algebra. A representation of $\mathcal{A}$ is a triplet $(M,\rho,\beta)$, where
$M$ is a $\Gamma$-graded vector space, $\beta \in End(M)_{0}$ and $\rho:\mathcal{A} \longrightarrow End(M) $ is an even linear map satisfying
\begin{equation}\label{reprecolor}
    \rho([x,y])\circ \beta=\rho(\alpha(x))\circ\rho(y)-\varepsilon(x,y)\rho(\alpha(y))\circ\rho(x),~~\forall~x,y \in \mathcal{H}(\mathcal{A}).
\end{equation}
\end{df}
Now, we discuss the adjoint representations of a color Hom-Lie algebra.
\begin{prop} Let $(\mathcal{A},[.,.],\varepsilon,\alpha)$ be a color Hom-Lie algebra and $ad:\mathcal{A}\longrightarrow End(\mathcal{A}) $ be an
operator defined for $x \in \mathcal{H}(\mathcal{A})$ by $ad(x)(y)=[x,y]$. Then $(\mathcal{A},ad,\alpha)$ is a representation of $\mathcal{A}$.
\end{prop}
\begin{proof} Since $\mathcal{A}$ is color Hom-Lie algebra, the $\varepsilon$-Hom-Jacobi condition on  $x,y,z \in \mathcal{H}(\mathcal{A})$ is
$$\varepsilon(z,x)[\alpha(x),[y,z]]+\varepsilon(y,z)[\alpha(z),[x,y]]+\varepsilon(x,y)[\alpha(y),[z,x]]=0$$
and may be written
$$ad([x,y])(\alpha(z))=ad(\alpha(x))(ad(y)(z))-\varepsilon(x,y)ad(\alpha(y))(ad(x)(z)).$$
Then the operator $ad$ satisfies
$$ad([x,y])\circ\alpha=ad(\alpha(x))\circ ad(y)-\varepsilon(x,y)ad(\alpha(y))\circ ad(x).$$
Therefore, it determines a representation of the color Hom-Lie algebra  $\mathcal{A}$.
\end{proof}
We call the representation defined in the previous Proposition the adjoint representation of the color Hom-Lie algebra  $\mathcal{A}$.
\begin{df} The $\alpha^{s}$-adjoint representation of the color Hom-Lie algebra $(\mathcal{A},[.,.],\varepsilon,\alpha)$, which we denote by $ad_{s}$, is defined by
$$ad_{s}(a)(x)=[\alpha^{s}(a),x],~~\forall~~ a,x \in \mathcal{H}(\mathcal{A}).$$
\end{df}
\begin{lem} With the above notations, we have  $(\mathcal{A},ad_{s}(.)(.),\alpha)$ is a representation of the color Hom-Lie algebra $(\mathcal{A},[.,.],\varepsilon,\alpha)$.
$$ad_{s}(\alpha(x))\circ \alpha=\alpha\circ ad_{s}(x).$$
$$ad_{s}([x,y])\circ \alpha=ad_{s}(\alpha(x))\circ ad_{s}(y)-\varepsilon(x,y)ad_{s}(\alpha(y))\circ ad_{s}(x).$$
\end{lem}
\begin{proof} First, the result follows from
\begin{eqnarray*}
  ad_{s}(\alpha(x))(\alpha(y)) &=&[\alpha^{s+1}(x),\alpha(y)]  \\
   &=& \alpha([\alpha^{s}(x),y]) \\
   &=&  \alpha\circ ad_{s}(x)(y)
\end{eqnarray*}
and
\begin{eqnarray*}
  ad_{s}([x,y])(\alpha(z)) &=&[\alpha^{s}([x,y]),\alpha(z)]
   = [[\alpha^{s}(x),\alpha^{s}(y)],\alpha(z)] \\
   &=&\varepsilon(x+y,z)\varepsilon(z,y)\varepsilon(z,x)[\alpha^{s+1}(x),[\alpha^{s}(y),z]\\
   &+&\varepsilon(x+y,z)\varepsilon(z,y)\varepsilon(x,y)[\alpha^{s+1}(y),[z,\alpha^{s}(x)] \\
   &=&[\alpha^{s+1}(x),[\alpha^{s}(y),z]+\varepsilon(x+y,z)[\alpha^{s+1}(y),[z,\alpha^{s}(x)]  \\
   &=& [\alpha^{s+1}(x),[\alpha^{s}(y),z]-\varepsilon(x,y)[\alpha^{s+1}(y),[\alpha^{s}(x),z] \\
   &=&[\alpha^{s+1}(x),ad_{s}(y)(z)]-\varepsilon(x,y)[\alpha^{s+1}(y),ad_{s}(x)(z)]  \\
   &=& ad_{s}(\alpha(x))\circ ad_{s}(y)(z)-\varepsilon(x,y)ad_{s}(\alpha(y))\circ ad_{s}(x)(z).
\end{eqnarray*}
Then $ad_{s}([x,y])\circ \alpha=ad_{s}(\alpha(x))\circ ad_{s}(y)-\varepsilon(x,y)ad_{s}(\alpha(y))\circ ad_{s}(x)$.
\end{proof}
The set of $n$-cochains on $\mathcal{A}$ with coefficients in $\mathcal{A}$ which we denote by $C^{n}_{\alpha}(\mathcal{A},\mathcal{A})$, is given by
$$C^{n}_{\alpha}(\mathcal{A},\mathcal{A})=\{f \in C^{n}(\mathcal{A},\mathcal{A}) ~:~f\circ \alpha^{\otimes n}=\alpha\circ f\}.$$
In particular, the set of $0$-cochains is given by
$$C^{0}_{\alpha}(\mathcal{A},\mathcal{A})=\{x \in \mathcal{H}(\mathcal{A})~:~\alpha(x)=x\}.$$
\begin{prop}\label{deriv} Associated to the $\alpha^{s}$-adjoint representation $ad_{s}$, of the color Hom-Lie algebra $(\mathcal{A},[.,.],\varepsilon,\alpha)$,
$D \in C^{1}_{\alpha,ad_{s}}(\mathcal{A},\mathcal{A})$ is $1$-cocycle if and only if $D$ is an $\alpha^{s+1}$-derivation of the color Hom-Lie algebra $(\mathcal{A},[.,.],\varepsilon,\alpha)$ of degree $\gamma$. (i.e. $D\in (Der_{\alpha^{s+1}}(\mathcal{A}))_{\gamma})$.
\end{prop}
\begin{proof} The conclusion follows directly from the definition of the coboundary $\delta$. $D$ is closed if and only if
$$\delta (D)(x,y)=-D([x,y])+\varepsilon(\gamma,x)[\alpha^{s+1}(x),D(y)]-\varepsilon(\gamma+x,y)[\alpha^{s+1}(y),D(x)]=0.$$
So $$D([x,y])=[D(x),\alpha^{s+1}(y)]+\varepsilon(\gamma,x)[\alpha^{s+1}(x),D(y)]$$
which implies that $D$ is an $\alpha^{s+1}$-derivation of $(\mathcal{A},[.,.],\varepsilon,\alpha)$ of degree $\gamma$.
\end{proof}
\subsubsection{The $\alpha^{-1}$-adjoint representation $ad_{-1}$}
\begin{prop} Associated to the $\alpha^{-1}$-adjoint representation $ad_{-1}$, we have
\begin{eqnarray*}
H^{0}(\mathcal{A},\mathcal{A})&=&C^{0}_{\alpha}(\mathcal{A},\mathcal{A})=\{x \in \mathcal{H}(\mathcal{A}) ~:~ \alpha(x)=x\}.\\
H^{1}(\mathcal{A},\mathcal{A})&=&Der_{\alpha^{0}}(\mathcal{A}).
\end{eqnarray*}
\end{prop}
\begin{proof} For any $0$-cochain $x \in C^{0}_{\alpha}(\mathcal{A},\mathcal{A})$, we have
$$\delta(x)(y)=\varepsilon(x,y)[\alpha^{-1}(y),x]=0,~~\forall~~ y \in \mathcal{H}(\mathcal{A}).$$
Therefore any $0$-cochain is closed. Thus, we have
$$H^{0}(\mathcal{A},\mathcal{A})=C^{0}_{\alpha}(\mathcal{A},\mathcal{A})=\{x \in \mathcal{H}(\mathcal{A}) ~:~ \alpha(x)=x\}.$$
Since there is not exact $1$-cochain, by Proposition \ref{deriv}, we have
$$H^{1}(\mathcal{A},\mathcal{A})=Der_{\alpha^{0}}(\mathcal{A}).$$
Let $w \in C^{2}_{\alpha}(\mathcal{A},\mathcal{A})$ be an even $\varepsilon$-skew-symmetric bilinear operator commuting with $\alpha$. Considering a t-parametrized family of bilinear operations
$$[x,y]_{t}=[x,y]+ tw(x,y).$$
Since $w$ commute with $\alpha$, $\alpha$ is a morphism with respect to the bracket $[.,.]_{t}$ for every $t$. If all bracket $[.,.]_{t}$ endow that $w$ generates a deformation of the color Hom-Lie algebra $(\mathcal{A},[.,.],\varepsilon,\alpha)$. By computing the $\varepsilon$-Hom-Jacobi condition of
 $[.,.]_{t}$, this is equivalent to
 \begin{eqnarray}\label{rep1}
   \circlearrowleft_{x,y,z}\varepsilon(z,x)\Big(w(\alpha(x),[y,z])+[\alpha(x),w(y,z)]\Big)&=&0,
 \end{eqnarray}
 \begin{eqnarray}\label{rep2}
   \circlearrowleft_{x,y,z}\varepsilon(z,x) w\Big(\alpha(x),w(y,z)\Big)&=&0.
 \end{eqnarray}
 Obviously, $(\ref{rep1})$ means that $w$ is an even $2$-cocycle with respect to the $\alpha^{-1}$-adjoint representation $ad_{-1}$.\\
 Furthermore, $(\ref{rep2})$ means that $w$ must itself defines a color Hom-Lie algebra structure on $\mathcal{A}$.
 \end{proof}
\subsubsection{The $\alpha^{0}$-adjoint representation $ad_{0}$}
\begin{prop} Associated to the $\alpha^{0}$-adjoint representation $ad_{0}$, we have
\begin{eqnarray*}
H^{0}(\mathcal{A},\mathcal{A})&=&\{x \in \mathcal{H}(\mathcal{A}) ~:~ \alpha(x)=x,~~[x,y]=0\}.\\
H^{1}(\mathcal{A},\mathcal{A})&=& \frac{Der_{\alpha}(\mathcal{A})}{Inn_{\alpha}(\mathcal{A})}.
\end{eqnarray*}
\end{prop}
\begin{proof} For any $0$-cochain, we have $d_{0}(x)(y)=[\alpha^{0}(x),y]=[x,y]$.\\
Therefore, the set of $0$-cocycle $Z^{0}(\mathcal{A},\mathcal{A})$ is given by
$$Z^{0}(\mathcal{A},\mathcal{A})=\{x \in C^{0}_{\alpha}(\mathcal{A},\mathcal{A}) ~:~[x,y]=0,~~\forall~~ y \in \mathcal{H}(\mathcal{A})\}.$$
As, $B^{0}(\mathcal{A},\mathcal{A})=\{0\}$, we deduce that $H^{0}(\mathcal{A},\mathcal{A})=\{x \in \mathcal{H}(\mathcal{A}) ~:~ \alpha(x)=x,~~[x,y]=0\}$\\
For any $f \in C^{1}_{\alpha}(\mathcal{A},\mathcal{A})$, we have
\begin{eqnarray*}
  \delta(f)(x,y) &=&-f([x,y])+\varepsilon(f,x)[\alpha(x),f(y)]-\varepsilon(f+x,y)[\alpha(y),f(x)]  \\
   &=& -f([x,y])+\varepsilon(f,x)[\alpha(x),f(y)]+[f(x),\alpha(y)].
\end{eqnarray*}
Therefore the set of $1$-cocycles is exactly the set of $\alpha$-derivations $Der_{\alpha}(\mathcal{A}).$\\
Furthermore, it is obvious that any exact $1$-coboundary is of the form of $[x,.]$ for some $x \in C^{0}_{\alpha}(\mathcal{A},\mathcal{A})$.\\
Therefore, we have $B^{1}(\mathcal{A},\mathcal{A})=Inn_{\alpha}(\mathcal{A})$. Which implies that $H^{1}(\mathcal{A},\mathcal{A})=\frac{Der_{\alpha}(\mathcal{A})}{Inn_{\alpha}(\mathcal{A})}.$
\end{proof}
\subsubsection{The coadjoint representation $\widetilde{ad}$}
In this subsection, we explore the dual representations and coadjoint representations of color Hom-Lie algebras.
Let $(\mathcal{A},[.,.],\varepsilon,\alpha)$ be a color Hom-Lie algebra and $(M,\rho,\beta)$ be a representation of $\mathcal{A}$. Let $M^{\ast}$ be the dual vector space of $M$. We define a linear map \\
$\widetilde{\rho}:\mathcal{A} \longrightarrow End(M^{\ast})$ by $\widetilde{\rho}(x)=-^{t}\rho(x)$.
Let $f \in  M^{\ast}, x , y \in \mathcal{H}(\mathcal{A})$ and $m \in M$. We compute the right hand side of the identity $(\ref{reprecolor})$.
\begin{align*}
&(\widetilde{\rho}(\alpha(x))\circ\widetilde{\rho}(y)-\varepsilon(x,y)\widetilde{\rho}(\alpha(y))\circ\widetilde{\rho}(x))(f)(m) \\ &=(\widetilde{\rho}(\alpha(x))\circ \widetilde{\rho}(y)(f)-\varepsilon(x,y)\widetilde{\rho}(\alpha(y))\circ\widetilde{\rho}(x)(f))(m)\\
&=-\widetilde{\rho}(y)(f)(\rho(\alpha(x)(m))+\varepsilon(x,y)\widetilde{\rho}(x)(f)(\rho(\alpha(y)(m))\\
&= f(\rho(y)\circ \rho(\alpha(x))(m))-\varepsilon(x,y) f(\rho(x)\circ \rho(\alpha(y))(m))\\
&= f(\rho(y)\circ \rho(\alpha(x))(m)-\varepsilon(x,y) \rho(x)\circ \rho(\alpha(y))(m)).
\end{align*}
On the other hand, we set that the twisted map for $\widetilde{\rho}$ is $\widetilde{\beta}=^{t}\beta$, the left hand side of $(\ref{reprecolor})$ writes
\begin{eqnarray*}
  \widetilde{\rho}([x,y])\circ \widetilde{\beta}(f)(m) &=& \widetilde{\rho}([x,y])(f\circ \beta)(m) \\
   &=& -f\circ \beta (\rho([x,y])(m)).
\end{eqnarray*}
Therefore, we have the following Proposition:
\begin{prop} Let $(\mathcal{A},[.,.],\varepsilon,\alpha)$ be a color Hom-Lie algebra and $(M,\rho,\beta)$ be a representation of $\mathcal{A}$. Let $M^{\ast}$ be the dual vector space of $M$. The triple $(M^{\ast},\widetilde{\rho},\widetilde{\beta})$, where $\widetilde{\rho}:\mathcal{A} \longrightarrow End(M^{\ast})$ is given by $\widetilde{\rho}(x)=-^{t}\rho(x)$, defines a representation of color Hom-Lie algebra $(\mathcal{A},[.,.],\varepsilon,\alpha)$ if and only if
\begin{equation}\label{rep coad}
    \beta \circ \rho([x,y])= \rho(x)\circ \rho(\alpha(y))- \varepsilon(x,y)\rho(y)\circ \rho(\alpha(x)).
\end{equation}
\end{prop}
We obtain the following characterization in the case of adjoint representation.
\begin{cor} Let $(\mathcal{A},[.,.],\varepsilon,\alpha)$ be a color Hom-Lie algebra and $(\mathcal{A},ad,\alpha)$ be the adjoint representation of $\mathcal{A}$, where $ad: \mathcal{A} \longrightarrow End(\mathcal{A})$. We set $\widetilde{ad}:\mathcal{A} \longrightarrow End(\mathcal{A}^{\ast})$ and $\widetilde{ad}(x)(f)=-f \circ ad(x)$. Then $(\mathcal{A}^{\ast},\widetilde{ad},\widetilde{\alpha})$ is a representation of $\mathcal{A}$ if and only if
 $$\alpha \circ ad([x,y])= ad(x)\circ ad(\alpha(y))- \varepsilon(x,y)ad(y)\circ ad(\alpha(x)),~~\forall~~ x,y \in \mathcal{H}(\mathcal{A}).$$
\end{cor}

\subsection{Example}\label{second group} Let $(sl^{c}_{2},\mathbb{C})$ be a color Lie algebra such that $sl^{c}_{2}=\oplus_{\gamma \in \Gamma}X_{\gamma}$ and $\Gamma=\mathbb{Z}_{2}\times\mathbb{Z}_{2}$, with
$$X_{(0,0)}=0,~~ X_{(1,0)}=e_{1},~~ X_{(0,1)}=e_{2},~~ X_{(1,1)}=e_{3}.$$
The algebra $(sl^{c}_{2},\mathbb{C})$ has a homogeneous basis $\{e_{1},e_{2},e_{3}\}$ with degree given by
$$|e_{1}|=\gamma_{1}=(1,0),~~ |e_{2}|=\gamma_{2}=(0,1),~~ |e_{3}|=\gamma_{3}=(1,1).$$
The bracket $[.,.]$ in $sl^{c}_{2}$ is given by
$$[e_{1},e_{2}]=e_{3},~~~~[e_{1},e_{3}]=e_{2},~~~~[e_{2},e_{3}]=e_{1}.$$
Then $(sl^{c}_{2},[.,.])$ is a color Lie algebra. According to Theorem $\ref{induced-color}$, the triple $(sl^{c}_{2},[.,.]_{\alpha}=\alpha \circ [.,.], \alpha)$ is a color Hom-Lie algebra such that the bracket $[.,.]_{\alpha}$ and the even linear map $\alpha$ are defined by
\begin{eqnarray*}
&&[e_{1},e_{2}]_{\alpha}=e_{3},\quad\quad \alpha(e_{1})=-e_{1}\\
&&[e_{1},e_{3}]_{\alpha}=-e_{2},\quad\alpha(e_{2})=-e_{2}\\
&&[e_{2},e_{3}]_{\alpha}=-e_{1},\quad\alpha(e_{3})=e_{3}.
\end{eqnarray*}
Let $\psi \in C^{1}_{ad}(sl^{c}_{2},sl^{c}_{2})$. The $2$-coboundary is defined by equation $(\ref{2-cocycle}).$
Now, suppose that $\psi$ is a $2$-cocycle of $sl^{c}_{2}$. Then $\psi$ satisfies
\begin{eqnarray}\label{f2cocy}
  \psi(\alpha(x),[y,z]_{\alpha}) &=& \psi([x,y]_{\alpha},\alpha(z))- \varepsilon(\gamma,x)[\alpha(x),\psi(y,z)]_{\alpha}+\varepsilon(\gamma+x,y)[\alpha(y),\psi(x,z)]_{\alpha}\nonumber\\
  &-&\varepsilon(\gamma+x+y,z)[\alpha(z),\psi(x,y)]_{\alpha}-\varepsilon(y,z)\psi([x,z]_{\alpha},\alpha(y))
\end{eqnarray}
By plugging the following triples
$$(e_{1},e_{2},e_{3}),~~(e_{1},e_{2},e_{3}),~~(e_{2},e_{1},e_{3}),(e_{2},e_{3},e_{1}),$$
$$(e_{3},e_{1},e_{2}),~~(e_{3},e_{2},e_{1}),\cdots~~(e_{3},e_{3},e_{2}),(e_{3},e_{3},e_{3}).$$
respectively in $(\ref{f2cocy})$. \\
\textbf{\underline{Case $1$}}$:$ If $\gamma=\gamma_{1}=(1,0)$, we obtain$:$
\begin{eqnarray*}
\quad \bullet~~ Z^{2}_{\gamma_{1}}(sl^{c}_{2},sl^{c}_{2})&=&\{\psi:~\psi(e_{i},e_{i})=0,~\psi(e_{i},e_{j})=\psi(e_{j},e_{i}),~\forall~i\neq j,~i=1,2,3 \}\\
&=&\{\psi:~\psi(e_{1},e_{2})=a_{1}e_{2}+a_{2}e_{3},~ \psi(e_{1},e_{3})=a_{3}e_{1}+a_{4}e_{2}+a_{1}e_{3},~\psi(e_{2},e_{3})=a_{3}e_{2}\},\\
\quad \bullet~~ B^{2}_{\gamma_{1}}(sl^{c}_{2},sl^{c}_{2})&=&\{\delta f :~\delta f(e_{i},e_{i})=0,~\delta f(e_{1},e_{2})=a_{2}e_{3},~~
 \delta f(e_{1},e_{3})=a_{4}e_{2},~~\delta f(e_{2},e_{3})=0,~~\forall~i=1,2,3\}.
\end{eqnarray*}
Then \begin{eqnarray*}&&\ \ H^{2}_{\gamma_{1}}(sl^{c}_{2},sl^{c}_{2})=\{\psi:~\psi(e_{1},e_{2})=a_{1}e_{2},~\psi(e_{1},e_{3})=a_{3}e_{1}+a_{1}e_{3}\}.\end{eqnarray*}
\textbf{\underline{Case $2$}}$:$ If $\gamma=\gamma_{2}=(0,1)$
\begin{eqnarray*}
\bullet~~ Z^{2}_{\gamma_{2}}(sl^{c}_{2},sl^{c}_{2})&=&\{\psi:~\psi(e_{i},e_{i})=0,~\psi(e_{i},e_{j})=\psi(e_{j},e_{i}),~\forall~i\neq j,~i=1,2,3 \}\\
 &=&\{\psi:~\psi(e_{i},e_{i})=0,~\psi(e_{1},e_{2})=a_{2}e_{3},~\psi(e_{1},e_{3})=0,~\psi(e_{2},e_{3})=a_{5}e_{1}\},\\
\bullet~~ B^{2}_{\gamma_{2}}(sl^{c}_{2},sl^{c}_{2})&=&\{\delta f :~\delta f(e_{i},e_{i})=0,~\delta f(e_{1},e_{2})=a_{2}e_{3},~~
 \delta f(e_{1},e_{3})=0,~~\delta f(e_{2},e_{3})=a_{5}e_{1},~~\forall~i=1,2,3\}
\end{eqnarray*}
Then $H^{2}_{\gamma_{2}}(sl^{c}_{2},sl^{c}_{2})=\{0\}.$\\
\textbf{\underline{Case $3$}}$:$ If $\gamma=\gamma_{3}=(1,1)$
\begin{eqnarray*}
\bullet~~ Z^{2}_{\gamma_{3}}(sl^{c}_{2},sl^{c}_{2})&=&\{\psi:~\psi(e_{i},e_{i})=0,~\psi(e_{i},e_{j})=\psi(e_{j},e_{i}),~\forall~i\neq j,~i=1,2,3\} \\
&=&\{\psi:~\psi(e_{i},e_{i})=0,~\psi(e_{1},e_{2})=0,~\psi(e_{1},e_{3})=a_{1}e_{1}+a_{2}e_{2},~\psi(e_{2},e_{3})=a_{3}e_{1}-a_{1}e_{2}\},\\
\bullet~~ B^{2}_{\gamma_{3}}(sl^{c}_{2},sl^{c}_{2})&=&\{\delta f :~\delta f(e_{i},e_{i})=0,~\delta f(e_{1},e_{2})=0,~
\delta f(e_{1},e_{3})=a_{1}e_{1}+a_{2}e_{2},~~
 \delta f(e_{2},e_{3})=a_{3}e_{1}-a_{1}e_{2},\\&&~\forall~i=1,2,3\}
\end{eqnarray*}
Then $H^{2}_{\gamma_{3}}(sl^{c}_{2},sl^{c}_{2})=\{0\}$.\\
So
\begin{eqnarray*}
H^{2}(sl^{c}_{2},sl^{c}_{2})&=&H^{2}_{\gamma_{1}}(sl^{c}_{2},sl^{c}_{2})\oplus H^{2}_{\gamma_{2}}(sl^{c}_{2},sl^{c}_{2})
\oplus H^{2}_{\gamma_{3}}(sl^{c}_{2},sl^{c}_{2})
 =\{\psi:~\psi(e_{1},e_{2})=a_{1}e_{2},~\psi(e_{1},e_{3})=a_{3}e_{1}+a_{1}e_{3}\}.
\end{eqnarray*}

\begin{center}\section{Formal deformations of color Hom-Lie algebras}\end{center}
\subsection{Formal deformations of color Hom-Lie algebras}
\begin{df} Let $(\mathcal{A},[.,.],\varepsilon,\alpha)$ be a color Hom-Lie algebra. A one parameter formal deformation of
 $\mathcal{A}$ is given by  $\mathbb{K}[[t]]$-bilinear  map
$[.,.]_{t} : \mathcal{A}[[t]]\times \mathcal{A}[[t]]  \longrightarrow  \mathcal{A}[[t]] $ of the form $[.,.]_{t}=\sum\limits_{i\geq0}t^{i}[.,.]_{i}$
where $[.,.]_{i}$ is an even $\mathbb{K}$-bilinear map $ [.,.]_{i} : \mathcal{A}[[t]]\times \mathcal{A}[[t]]  \longrightarrow  \mathcal{A}[[t]]$ (extended to be even $\mathbb{K}[[t]]$-bilinear) and satisfying for all $x,y,z \in \mathcal{H}(\mathcal{A})$ the following conditions
\begin{equation}\label{jh1}
~~[x,y]_{t}=-\varepsilon(x,y)[y,x]_{t},
\end{equation}
\begin{equation}\label{jh}
    \circlearrowleft_{x,y,z}\varepsilon(z,x)[\alpha(x),[y,z]_{t}]_{t}=0.
\end{equation}
The deformation is said to be of order k if $[.,.]_{t} =\sum\limits_{i\geq0}^{k}t^{i}[.,.]_{i}.$
\end{df}
\begin{rem} The $\varepsilon$-skew symmetry of $[.,.]_{t}$ is equivalent to the $\varepsilon$-skew symmetry of $[.,.]_{i}$ for $i \in \mathbb{Z}_{\geq 0}.$
\end{rem}
Condition $(\ref{jh} )$ is called  deformation equation of the color Hom-Lie algebra and it is equivalent to
$$\circlearrowleft_{x,y,z}\sum_{i,j,k \geq 0}\varepsilon(z,x)t^{i+j}[\alpha(x),[y,z]_{i}]_{j}=0$$
i.e $$\circlearrowleft_{x,y,z}\sum_{i,s\geq 0}\varepsilon(z,x)t^{s}[\alpha(x),[y,z]_{i}]_{s-i}=0$$
or $$\sum_{s \geq 0}t^{s}\circlearrowleft_{x,y,z}\sum_{i,s\geq 0}\varepsilon(z,x)[\alpha(x),[y,z]_{i}]_{s-i}=0$$
which is equivalent to the following infinite system
\begin{equation}\label{jjjj}
  \circlearrowleft_{x,y,z}\sum_{i,k\geq 0}\varepsilon(z,x)[\alpha(x),[y,z]_{i}]_{s-i}=0,~~\forall~~ s=0,1,2 \cdots
\end{equation}
In particular, for $s=0$, we have $\circlearrowleft_{x,y,z}\varepsilon(z,x)[\alpha(x),[y,z]_{0}]_{0}=0$, which  is the $\varepsilon$-Hom-Jacobi condition of
$\mathcal{A}$.\\
The equation for $s=1$, leads to $\delta^{2}([.,.]_{1})(x,y,z)=0$. Then $[.,.]_{1}$ is a 2-cocycle.\\
For $s\geq 2$, the identity $(\ref{jjjj})$ is equivalent to
$$\delta^{2}([.,.]_{s})(x,y,z)=-\sum_{p+q=s}\circlearrowleft_{x,y,z}\varepsilon(z,x)[\alpha(x),[y,z]_{q}]_{p}=0.$$
\subsection{Equivalent and trivial deformations}
\begin{df} Let $(\mathcal{A},[.,.],\varepsilon,\alpha)$ be a multiplicative color Hom-Lie algebra. Given two deformations $\mathcal{A}_{t}=(\mathcal{A},[.,.]_{t},\varepsilon,\alpha)$ and $\mathcal{A}^{'}_{t}=(\mathcal{A},[.,.]^{'}_{t},\varepsilon,\alpha^{'})$ of $\mathcal{A}$ where $[.,.]_{t}=\sum\limits_{i\geq0}^{k}t^{i}[.,.]_{i}$ and $[.,.]^{'}_{t}=\sum\limits_{i\geq0}^{k}t^{i}[.,.]^{'}_{i}$ with $[.,.]_{0}=[.,.]^{'}_{0}=[.,.]$ . We say that $\mathcal{A}_{t}$ and $\mathcal{A}^{'}_{t}$ are equivalent if there exists a formal automorphism
$\phi_{t}:\mathcal{A}[[t]] \longrightarrow  \mathcal{A}[[t]]$ that may be written in the form $\phi_{t}=\sum\limits_{i\geq0}\phi_{i}t^{i}$, where $\phi_{i}\in End(\mathcal{A})_{0}$ and $\phi_{0}=Id$ such that
\begin{eqnarray}\label{dAN}
    \phi_{t}([x,y]_{t})&=&[\phi_{t}(x),\phi_{t}(y)]^{'}_{t}\\
\phi_{t}(\alpha(x))&=&\alpha^{'}(\phi_{t}(x)).\nonumber
\end{eqnarray}
\end{df}
A deformation $\mathcal{A}_{t}$ of $\mathcal{A}$ is said to be trivial if and only if $\mathcal{A}_{t}$ is  equivalent to $\mathcal{A}$. Viewed as an algebra on  $\mathcal{A}[[t]]$.
\begin{df} Let $(\mathcal{A},[.,.],\varepsilon,\alpha)$ be a color Hom-Lie algebra and $[.,.]_{1}\in Z^{2}(\mathcal{A},\mathcal{A})$.\\
The $2$-cocycle $[.,.]_{1}$ is said to be integrable if there exists a family $([.,.]_{i})_{i \geq 0}$ such that $[.,.]_{t}=\sum\limits_{i\geq0}t^{i}[.,.]_{i}$
defines a formal deformation $\mathcal{A}_{t}=(\mathcal{A},[.,.]_{t},\varepsilon,\alpha)$ of $\mathcal{A}$.
\end{df}
\begin{thm} Let $(\mathcal{A},[.,.],\varepsilon,\alpha)$ be a color Hom-Lie algebra and $\mathcal{A}_{t}=(\mathcal{A},[.,.]_{t},\varepsilon,\alpha)$ be a one parameter formal deformation of $\mathcal{A}$, where $[.,.]_{t}=\sum\limits_{i\geq 0}t^{i}[.,.]_{i}$. Then \begin{enumerate}
\item The first term $[.,.]_{1}$ is a $2$-cocycle with respect to the cohomology of $(\mathcal{A},[.,.],\varepsilon,\alpha)$.
\item there exists an equivalent deformation $\mathcal{A}^{'}_{t}=(\mathcal{A},[.,.]^{^{'}}_{t},\varepsilon,\alpha^{'})$, where $[.,.]^{'}_{t}=\sum\limits_{i\geq 0}t^{i}[.,.]^{'}_{i}$ such that $[.,.]^{'}_{1} \in Z^{2}(\mathcal{A},\mathcal{A})$ and
    $[.,.]^{'}_{1}\not \in B^{2}(\mathcal{A},\mathcal{A})$.\\
    Moreover, if $H^{2}(\mathcal{A},\mathcal{A})=0$, then every formal deformation is trivial.
    \end{enumerate}
\end{thm}
\subsection{Deformation by composition}
In the sequel, we give a procedure of deforming color Lie algebras into color Hom-Lie algebras using the following Proposition$:$
\begin{prop} Let $(\mathcal{A},[.,.],\varepsilon)$ be a color Lie algebra and $\alpha_{t}$ be an even algebra endomorphism of the form $\alpha_{t}=\alpha_{0}+\sum\limits_{i\geq1}^{p}t^{i}\alpha_{i}$, where $\alpha_{i}$ are linear maps on $\mathcal{A}$,~~$t$ is a parameter in $\mathbb{K}$ and $p$ is an  integer. Let $[.,.]_{t}=\alpha_{t}\circ [.,.]$, then $(\mathcal{A},[.,.]_{t},\varepsilon,\alpha_{t})$ is a color Hom-Lie algebra which is a deformation of the color Lie algebra viewed as a color Hom-Lie algebra $(\mathcal{A},[.,.],\varepsilon,Id)$.\\
Moreover, the $n$th derived Hom-algebra 
$$\mathcal{A}_{t}^{n}=(\mathcal{A},[.,.]_{t}^{(n)}=\alpha_{t}^{2^{n}-1}\circ[.,.]_{t},\varepsilon,\alpha_{t}^{2^{n}})$$
is a deformation of $(\mathcal{A},[.,.],\varepsilon,Id)$.
\end{prop}
\begin{proof} The first assertion follows from Theorem $\ref{induced-color}$. In particular for an infinitesimal deformation of the identity  $\alpha_{t}=Id+t\alpha_{1}$, we have $[.,.]_{t}=[.,.]+ t\alpha_{1}\circ[.,.]$.\\
The proof of the $\varepsilon$-Hom-Jacobi condition of the nth derived Hom-algebra follows from  Theorem $\ref{abdaw}$. In case $n=1$ and $\alpha_{t}=Id+t\alpha_{1}$ the bracket is
\begin{eqnarray*}
   [.,.]^{(1)}&=& Id+t\alpha_{1}\circ Id+t\alpha_{1}\circ[.,.] \\
   &=& [.,.]+2t\alpha_{1}\circ [.,.]+t^{2}\alpha_{1}\circ [.,.]
\end{eqnarray*}
and the twist map is $\alpha_{t}^{2}=(Id+t\alpha_{1})^{2}=Id+2t\alpha_{1}+t^{2}\alpha_{1}$. Therefore we get another deformation of the color Lie algebra viewed as a color Hom-Lie algebra $(\mathcal{A},[.,.],\varepsilon,Id)$. The proof in the general case is similar.
\end{proof}
\begin{rem} More generally if $(\mathcal{A},[.,.],\varepsilon,\alpha)$ is  a multiplicative color Hom-Lie algebra where $\alpha$ may be written of the form $\alpha=Id+t\alpha_{1}$, then the nth derived Hom-algebra 
$$\mathcal{A}_{t}^{n}=(\mathcal{A},[.,.]^{(n)}=\alpha^{n}\circ[.,.],\varepsilon,\alpha^{n+1})$$
gives a one parameter formal deformation of $(\mathcal{A},[.,.],\varepsilon,\alpha)$. But for any $\alpha$ one obtains just new color Hom-Lie algebra.
\end{rem}
\section{Generalized $\alpha^{k}$-derivations of color Hom-Lie algebras}
The purpose of this section is to study the homogeneous $\alpha^{k}$-generalized derivations and homogeneous $\alpha^{k}$-centroid of color Hom-Lie algebras generalizing the homogeneous generalized derivations discussed in \cite{LIn ni}. In Proposition $\ref{DERhomJORDAN}$ we prove that the $\alpha$-derivation of color Hom-Lie algebras gives rise to a Hom-Jordan color algebras.\\
We need the following definitions$:$
\begin{df} Let $Pl_{\gamma}(\mathcal{A})=\{D \in Hom(\mathcal{A},\mathcal{A}):~~D(\mathcal{A}_{\gamma})\subset \mathcal{A}_{\gamma+\mu}$ for all $\gamma,\mu \in \Gamma \}$.\\ Then $\Big(Pl(\mathcal{A})=\bigoplus_{\gamma \in \Gamma}Pl_{\gamma}(\mathcal{A}), [.,.],\alpha\Big)$ is a color Hom-Lie algebra with the color Lie bracket
$$[D_{\gamma},D_{\mu}]=D_{\gamma}\circ D_{\mu}-\varepsilon(\gamma,\mu)D_{\mu}\circ D_{\gamma}$$
for all $D_{\gamma},D_{\mu} \in \mathcal{H}(Pl(\mathcal{A}))$ and with $\alpha:\mathcal{A} \longrightarrow \mathcal{A}$ is an even homomorphism.\\
A homogeneous $\alpha^{k}$-derivation of degree $\gamma$ of $\mathcal{A}$ is an endomorphism $D \in Pl_{\gamma}(\mathcal{A})$ such that
$$[D,\alpha]=0,$$
$$D([x,y])=[D(x),\alpha^{k}(y)]+\varepsilon(\gamma,x)[\alpha^{k}(x),D(y)]$$
for all $x,y$ in $\mathcal{A}$.\\
We denote the set of all homogeneous $\alpha^{k}$-derivations of degree $\gamma$ of $\mathcal{A}$ by $Der^{\gamma}_{\alpha^{k}}(\mathcal{A})$. The space
$$Der(\mathcal{A})=\bigoplus_{k\geq 0}Der_{\alpha^{k}}(\mathcal{A})$$
provided with the color-commutator is a color Lie algebra. Indeed, the fact that $Der_{\alpha^{k}}(\mathcal{A})$ is $\Gamma$-graded implies that $Der(\mathcal{A})$ is $\Gamma$-graded
$$(Der(\mathcal{A}))_{\gamma}=\bigoplus_{k\geq 0}(Der_{\alpha^{k}}(\mathcal{A}))_{\gamma},~~\forall~~ \gamma \in \Gamma.$$
\end{df}
\begin{df}\begin{enumerate} \item An endomorphism $D \in Pl_{\gamma}(\mathcal{A})$ is said to be a homogeneous generalized\\ $\alpha^{k}$-derivation of degree $\gamma$ of $\mathcal{A}$, if there exist two endomorphisms $D^{'},D^{''} \in Pl_{\gamma}(\mathcal{A})$ such that
$$[D,\alpha]=0,~~[D^{'},\alpha]=0,~~[D^{''},\alpha]=0$$
\begin{equation}\label{generder}
    D^{''}([x,y])=[D(x),\alpha^{k}(y)]+\varepsilon(\gamma,x)[\alpha^{k}(x),D^{'}(y)]
\end{equation}
for all $x,y$ in $\mathcal{A}$.\\
We denote the set of all homogeneous generalized $\alpha^{k}$-derivations of degree $\gamma$ of $\mathcal{A}$ by $GDer^{\gamma}_{\alpha^{k}}(\mathcal{A})$. The space
$$GDer(\mathcal{A})=\bigoplus_{k\geq 0}GDer_{\alpha^{k}}(\mathcal{A}).$$
\item We call $D \in Pl_{\gamma}(\mathcal{A})$ a homogeneous $\alpha^{k}$-quasi-derivation of degree $\gamma$ of $\mathcal{A}$,
if there exists an endomorphism $D^{'} \in Pl_{\gamma}(\mathcal{A})$ such that
$$[D,\alpha]=0,~~[D^{'},\alpha]=0$$
\begin{equation}\label{quasider}
    D^{'}([x,y])=[D(x),\alpha^{k}(y)]+\varepsilon(\gamma,x)[\alpha^{k}(x),D(y)]
\end{equation}
for all $x,y \in \mathcal{A}$.\\
We denote the set of all homogeneous $\alpha^{k}$-quasi-derivations of degree $\gamma$ of $\mathcal{A}$ by $QDer^{\gamma}_{\alpha^{k}}(\mathcal{A})$. The space
$$QDer(\mathcal{A})=\bigoplus_{k\geq 0}QDer_{\alpha^{k}}(\mathcal{A}).$$
\item If $C(\mathcal{A})=\bigoplus_{k\geq 0}C^{\gamma}_{\alpha^{k}}(\mathcal{A}),~~\forall~~ \gamma \in \Gamma,$ with $C^{\gamma}_{\alpha^{k}}(\mathcal{A})$ consisting of
$D \in Pl_{\gamma}(\mathcal{A})$ satisfying
\begin{equation}\label{centr}
    D([x,y])=[D(x),\alpha^{k}(y)]=\varepsilon(\gamma,x)[\alpha^{k}(x),D(y)]
\end{equation}
for all $x,y$ in $\mathcal{A}$, then $C(\mathcal{A})$ is called the $\alpha^{k}$-centroid of $\mathcal{A}$.\\
We denote the set of all homogeneous $\alpha^{k}$-centroid of degree $\gamma$ of $\mathcal{A}$ by $C^{\gamma}_{\alpha^{k}}(\mathcal{A})$.
\end{enumerate}
\end{df}
\begin{prop} Let $(\mathcal{A},[.,.],\varepsilon,\alpha)$ be a  multiplicative color Hom-Lie algebra.\\ If $D_{\gamma} \in GDer_{\alpha^{k}}(\mathcal{A})$ and
$\Delta_{\gamma^{'}} \in C_{\alpha^{k'}}(\mathcal{A})$, then $\Delta_{\gamma^{'}} D_{\gamma} \in GDer_{\alpha^{k+k^{'}}}(\mathcal{A})$ is of degree $(\gamma+\gamma^{'})$.
\end{prop}
\begin{proof} Let $D_{\gamma} \in GDer_{\alpha^{k}}(\mathcal{A})$. Then for all $x,y \in \mathcal{H}(\mathcal{A})$,
there exist $D^{'}_{\gamma},D^{''}_{\gamma} \in Pl_{\gamma}(\mathcal{A})$ such that $$D^{''}_{\gamma}([x,y])=[D_{\gamma}(x),\alpha^{k}(y)]+\varepsilon(\gamma,x)[\alpha^{k}(x),D^{'}_{\gamma}(y)].$$
Now, let $\Delta_{\gamma^{'}} \in C_{\alpha^{k'}}(\mathcal{A})$ then we have$:$
\begin{align*}
\Delta_{\gamma^{'}}D^{''}_{\gamma}([x,y])
&=\Delta_{\gamma^{'}}([D_{\gamma}(x),\alpha^{k}(y)]+\varepsilon(\gamma,x)[\alpha^{k}(x),D^{'}_{\gamma}(y)])\\
&=[\Delta_{\gamma^{'}}D_{\gamma}(x),\alpha^{k+k'}(y)]+\varepsilon(\gamma+\gamma^{'},x)[\alpha^{k+k'}(x),\Delta_{\gamma^{'}}D^{'}_{\gamma}(y)].
\end{align*}
Then $\Delta_{\gamma^{'}} D_{\gamma} \in GDer_{\alpha^{k+k^{'}}}(\mathcal{A})$ of degree $(\gamma+\gamma^{'})$.
\end{proof}
\begin{prop} Let $D_{\gamma} \in C_{\alpha^{k}}(\mathcal{A})$, then $D_{\gamma}$ is an $\alpha^{k}$-Quasi-derivation of $\mathcal{A}$.
\end{prop}
\begin{proof} Let $x,y \in \mathcal{H}(\mathcal{A})$, we have
\begin{align*}
[D_{\gamma}(x),\alpha^{k}(y)]+ \varepsilon(\gamma,x)[\alpha^{k}(x),D_{\gamma}(y)]
&=[D_{\gamma}(x),\alpha^{k}(y)]+[D_{\gamma}(x),\alpha^{k}(y)]\\
&=2 D_{\gamma}([x,y])\\
&=D^{"}_{\gamma}([x,y]).
\end{align*}
Then $D_{\gamma}$ is an $\alpha^{k}$-Quasi-derivation of degree $\gamma$ of $\mathcal{A}$.
\end{proof}
\begin{df} If $QC(\mathcal{A})=\bigoplus_{\gamma \in \Gamma}QC^{\gamma}_{\alpha^{k}}(\mathcal{A})$ and $QC^{\gamma}_{\alpha^{k}}(\mathcal{A})$ consisting of
$D \in Pl_{\gamma}(\mathcal{A})$ such that for all $x,y$ in $\mathcal{A}$ $$[D(x),\alpha^{k}(y)]=\varepsilon(\gamma,x)[\alpha^{k}(x),D(y)],$$ then $QC(\mathcal{A})$ is called the $\alpha^{k}$-quasi-centroid of $\mathcal{A}$.
\end{df}
\begin{prop} Let $D_{\gamma} \in \mathcal{H}(QC_{\alpha^{k}}(\mathcal{A})) $ and $D_{\mu} \in \mathcal{H}(QC_{\alpha^{k^{'}}}(\mathcal{A}))$. Then $[D_{\gamma},D_{\mu}]$ is an $\alpha^{k+k^{'}}$-generalized derivation of degree $(\gamma+\mu).$
\end{prop}
\begin{proof} Assume that $D_{\gamma} \in \mathcal{H}(QC_{\alpha^{k}}(\mathcal{A})),D_{\mu}\in \mathcal{H}(QC_{\alpha^{k^{'}}}(\mathcal{A})) $. then for all $x,y \in \mathcal{H}(\mathcal{A})$, we have
 $$[D_{\gamma}(x),\alpha^{k}(y)]=\varepsilon(\gamma,x)[\alpha^{k}(x),D_{\gamma}(y)]$$
and  $$[D_{\mu}(x),\alpha^{k^{'}}(y)]=\varepsilon(\mu,x)[\alpha^{k^{'}}(x),D_{\mu}(y)].$$
Hence
\begin{align*}
[[D_{\gamma},D_{\mu}](x),\alpha^{k+k^{'}}(y)]
&=[(D_{\gamma}\circ D_{\mu}-\varepsilon(\gamma,\mu)D_{\mu}\circ D_{\gamma})(x),\alpha^{k+k^{'}}(y)]\\
&=[D_{\gamma}\circ D_{\mu}(x),\alpha^{k+k^{'}}(y)]-\varepsilon(\gamma,\mu)[D_{\mu}\circ D_{\gamma}(x),\alpha^{k+k^{'}}(y)]\\
&=\varepsilon(\gamma+\mu,x)[\alpha^{k+k^{'}}(x),D_{\gamma}\circ D_{\mu}(y)]-\varepsilon(\gamma,\mu)\varepsilon(\gamma+\mu,x)[\alpha^{k+k^{'}}(x),D_{\mu}\circ D_{\gamma}(y)]\\
&=\varepsilon(\gamma+\mu,x)[\alpha^{k+k^{'}}(x),[D_{\gamma},D_{\mu}](y)]+[[D_{\gamma},D_{\mu}](x),\alpha^{k+k^{'}}(y)],
\end{align*}
which implies that $[[D_{\gamma},D_{\mu}](x),\alpha^{k+k^{'}}(y)]+[[D_{\gamma},D_{\mu}](x),\alpha^{k+k^{'}}(y)]=0.$\\
Then $[D_{\gamma},D_{\mu}] \in GDer_{\alpha^{k+k^{'}}}(\mathcal{A})$ and is  of degree $(\gamma+\mu).$
\end{proof}
\subsection{Hom-Jordan color algebras and Derivations}
\begin{df} Let $(\mathcal{A},\mu,\alpha)$ be a Hom-color algebra.\begin{enumerate}
\item The \textsf{Hom-associator} of $\mathcal{A}$  is the trilinear map $as_{\alpha}:\mathcal{A}\times \mathcal{A}\times \mathcal{A}\longrightarrow \mathcal{A}$ defined as
\begin{equation}\label{ass}
    as_{\alpha}=\mu \circ (\mu\otimes \alpha- \alpha \otimes \mu).
\end{equation}
In terms of elements, the map $as_{\alpha}$ is given by
$$as_{\alpha}(x,y,z)=\mu(\mu(x,y),\alpha(z))-\mu(\alpha(x),\mu(y,z))$$
for all $x,y,z$ in $\mathcal{H}(\mathcal{A})$.
\item Let $\mathcal{A}$ be a Hom-algebra over a field $\mathbb{K}$ of characteristic $\neq 2$ with an even bilinear multiplication $\circ$. If $\mathcal{A}$ is graded by the abelian group $\Gamma$, $\varepsilon: \Gamma \times \Gamma \longrightarrow \mathbb{K}^{\ast}$ and $\alpha:\mathcal{A}\longrightarrow \mathcal{A}$ be an even linear map, then  $(\mathcal{A},\circ,\varepsilon,\alpha)$ is a Hom-Jordan color algebra if the identities
$$ (HCJ1):~~x \circ y= \varepsilon(x,y) y \circ x$$
$$(HCJ2):~~\varepsilon(w,x+z) as_{\alpha}(x \circ y, \alpha(z),\alpha(w))+\varepsilon(x,y+z) as_{\alpha}(y \circ w, \alpha(z),\alpha(x))$$
$$+\varepsilon(y,w+z) as_{\alpha}(w \circ x, \alpha(z),\alpha(y))=0$$
are satisfied for all $x,y,z$ and $w$ in $\mathcal{H}(\mathcal{A}).$
\end{enumerate}
\end{df}
The identity $(HCJ2)$ is called the \textsf{color Hom-Jordan identity}.\\
Observe that when $\alpha=Id$, the color Hom-Jordan identity $(HCJ2)$ reduces to the usual color Jordan identity.
\begin{prop}\label{DERhomJORDAN} Let $(\mathcal{A},[.,.],\varepsilon,\alpha)$ be a multiplicative color Hom-Lie algebra, with the operation
$D_{\gamma}\bullet D_{\mu}=D_{\gamma}\circ D_{\mu}-\varepsilon(\gamma,\mu)D_{\gamma}\circ D_{\mu}$ for all $\alpha$-derivations $D_{\gamma}, D_{\mu} \in \mathcal{H}(Pl(\mathcal{A}))$, the triple $(Pl(\mathcal{A}),\bullet,\varepsilon,\alpha)$ is a Hom-Jordan color algebra.
\end{prop}
\begin{proof} Assume that $D_{\lambda},D_{\theta}, D_{\mu},D_{\gamma} \in \mathcal{H}(Pl(\mathcal{A}))$, we have
\begin{eqnarray*}
 D_{\lambda}\bullet D_{\theta}&=& D_{\lambda}\circ D_{\theta}-\varepsilon(\lambda,\theta)D_{\theta}\circ D_{\lambda}  \\
   &=&\varepsilon(\lambda,\theta)(D_{\theta}\circ D_{\lambda}-\varepsilon(\lambda,\theta)D_{\lambda}\circ D_{\theta})  \\
   &=&\varepsilon(\lambda,\theta)D_{\theta}\bullet D_{\lambda}.
\end{eqnarray*}
 Since
\begin{align*}
&((D_{\lambda}\bullet D{\theta})\bullet \alpha( D_{\mu}))\bullet \alpha^{2}(D_{\gamma})\\
&=D_{\lambda} D_{\theta} \alpha( D_{\mu}) \alpha^{2}(D_{\gamma})+\varepsilon(\lambda+\theta+\mu,\gamma)\alpha^{2}(D_{\gamma})D_{\lambda} D_{\theta}
\alpha(D_{\mu})\\
&+\varepsilon(\lambda+\theta,\mu)\alpha(D_{\mu})D_{\lambda} D_{\theta}\alpha^{2}(D_{\gamma})+\varepsilon(\lambda+\theta,\mu)\varepsilon(\lambda+\theta+\mu,\gamma)\alpha^{2}(D_{\gamma})\alpha(D_{\mu})D_{\lambda} D_{\theta}\\&+\varepsilon(\lambda,\theta)_D{\theta}D_{\lambda}\alpha(D_{\mu})\alpha^{2}(D_{\gamma})
+\varepsilon(\lambda,\theta)\varepsilon(\lambda+\theta+\mu,\gamma)\alpha^{2}(D_{\gamma})D_{\theta}D_{\lambda}\alpha(D_{\mu})\\
&+\varepsilon(\lambda,\theta)\varepsilon(\lambda+\theta,\mu)\alpha(D_{\mu})D_{\theta}D_{\lambda}\alpha^{2}(D_{\gamma})
+\varepsilon(\lambda,\theta)\varepsilon(\lambda+\theta,\mu)\varepsilon(\lambda+\theta+\mu,\gamma)\alpha^{2}(D_{\gamma})\alpha(D_{\mu})D_{\theta}D_{\lambda}
\end{align*}
and
\begin{align*}
\alpha(D_{\lambda}\bullet D_{\theta})\bullet (\alpha(D_{\mu})\bullet\alpha(D_{\gamma}))
&=\alpha(D_{\lambda} D_{\theta})\alpha(D_{\mu})\alpha(D_{\gamma})+\varepsilon(\lambda+\theta,\mu+\gamma)\alpha(D_{\mu})\alpha(D_{\gamma})\alpha(D_{\lambda} D_{\theta})\\
&+\varepsilon(\mu,\gamma)\alpha(D_{\lambda} D_{\theta})\alpha(D_{\gamma})\alpha(D_{\mu})\\
&+\varepsilon(\mu,\gamma)\varepsilon(\lambda+\theta,\mu+\gamma)\alpha(D_{\gamma})\alpha(D_{\mu})\alpha(D_{\lambda}D_{\theta})\\
&+\varepsilon(\lambda,\theta)\alpha(D_{\theta} D_{\lambda})\alpha(D_{\mu})\alpha(D_{\gamma})\\
&+\varepsilon(\lambda,\theta)\varepsilon(\lambda+\theta,\mu+\gamma)\alpha(D_{\mu})\alpha(D_{\gamma})\alpha(D_{\theta} D_{\lambda})\\
&+\varepsilon(\lambda,\theta)\varepsilon(\mu,\theta)\alpha(D_{\theta} D_{\lambda})\alpha(D_{\gamma})\alpha(D_{\mu})\\
&+\varepsilon(\lambda,\theta)\varepsilon(\mu,\gamma)\varepsilon(\lambda+\theta,\mu+\gamma)\alpha(D_{\gamma})\alpha(D_{\mu})\alpha(D_{\theta} D_{\lambda}).
\end{align*}
We have
\begin{align*}
&\varepsilon(\gamma,\lambda+\mu) as_{\alpha}(D_{\lambda}\bullet D_{\theta}, \alpha(D_{\mu}),\alpha(D_{\gamma}))\\
&=\varepsilon(\gamma,\lambda+\mu)\Big( \varepsilon(\lambda+\theta+\mu,\gamma)\alpha^{2}(D_{\gamma})D_{\lambda} D_{\theta}\alpha(D_{\mu})+\varepsilon(\lambda+\theta,\mu)\alpha(D_{\mu})D_{\lambda} D_{\theta}\alpha^{2}(D_{\gamma})\\ &+\varepsilon(\lambda,\theta)\varepsilon(\lambda+\theta+\mu,\gamma)\alpha^{2}(D_{\gamma})D_{\theta}D_{\lambda}\alpha(D_{\mu})
+\varepsilon(\lambda,\theta)\varepsilon(\lambda+\theta,\mu)\alpha(D_{\mu})D_{\theta}D_{\lambda}\alpha^{2}(D_{\gamma})\\
&-\varepsilon(\lambda+\theta,\mu+\gamma)\alpha(D_{\mu})\alpha(D_{\gamma})\alpha(D_{\theta} D_{\lambda})-\varepsilon(\mu,\gamma)\alpha(D_{\lambda} D_{\theta})\alpha(D_{\gamma})\alpha(D_{\mu})\\
&-\varepsilon(\lambda,\theta)\varepsilon(\lambda+\theta,\mu+\gamma)\alpha(D_{\mu})\alpha(D_{\gamma})\alpha(D_{\theta} D_{\lambda})-\varepsilon(\lambda,\theta)\varepsilon(\mu,\theta)\alpha(D_{\theta} D_{\lambda})\alpha(D_{\gamma})\alpha(D_{\mu})\Big).
\end{align*}
Therefore, we get
\begin{align*}
&\varepsilon(\gamma,\lambda+\mu) as_{\alpha}(D_{\lambda}\bullet D_{\theta}, \alpha(D_{\mu}),\alpha(D_{\gamma}))+
\varepsilon(\lambda,\theta+\mu) as_{\alpha}(D_{\theta}\bullet D_{\gamma}, \alpha(D_{\mu}),\alpha(D_{\gamma}))\\
&+\varepsilon(\theta,\gamma+\mu) as_{\alpha}(D_{\gamma}\bullet D_{\lambda}, \alpha(D_{\mu}),\alpha(D_{\theta}))=0,
\end{align*}
and so the statement holds.
\end{proof}
\begin{cor} Let $(\mathcal{A},[.,.],\varepsilon,\alpha)$ be a multiplicative color Hom-Lie algebra, with the operation
$$D_{\gamma}\bullet D_{\mu}=D_{\gamma}\circ D_{\mu}+\varepsilon(\gamma,\mu)D_{\mu}\circ D_{\gamma}$$ for all $D_{\gamma}, D_{\mu} \in \mathcal{H}(QC(\mathcal{A}))$, the quadruple $(QC(\mathcal{A}),\bullet,\varepsilon,\alpha)$ is a Hom-Jordan color algebra.
\end{cor}
\begin{proof} We need only to show that $D_{\gamma}\bullet D_{\mu} \in QC(\mathcal{A})$, for all $D_{\gamma}, D_{\mu} \in \mathcal{H}(QC(\mathcal{A})).$\\
Assume that $x,y \in \mathcal{H}(\mathcal{A})$, we have
\begin{align*}
 [D_{\gamma}\bullet D_{\mu}(x),\alpha^{k+s}(y)]
&=[D_{\gamma}\circ D_{\mu}(x),\alpha^{k+s}(y)]+\varepsilon(\gamma,\mu)[D_{\mu}\circ D_{\gamma}(x),\alpha^{k+s}(y)]\\
&=\varepsilon(\gamma,\mu+x)[D_{\mu}(x),D_{\gamma}(\alpha^{k+s}(y))]+\varepsilon(\mu,x)[D_{\gamma}(x),D_{\mu}(\alpha^{k+s}(y))] \\
&=\varepsilon(\gamma,\mu)\varepsilon(\gamma+\mu,x)[\alpha^{k+s}(x),D_{\mu}\circ D_{\gamma}(y)]+\varepsilon(\gamma+\mu,x)[\alpha^{k+s}(x),D_{\gamma}\circ D_{\mu}(y)] \\
&=\varepsilon(\gamma+\mu,x)[\alpha^{k+s}(x),D_{\gamma}\bullet D_{\mu}(y)].
\end{align*}
Hence $D_{\gamma}\bullet D_{\mu} \in QC(\mathcal{A}).$
\end{proof}


\end{document}